\documentclass[11pt,leqno]{article}
\usepackage{amssymb, amscd, amsmath, amsthm}

\allowdisplaybreaks

\def\ve{\varepsilon}
\def\intl{\int\limits}
\def\mod{\,\text{\rm mod}\;}
\newcommand{\beq}{\begin{equation}}
\newcommand{\eeq}{\end{equation}}

\newtheorem{Thm}{Theorem}
\newtheorem{Lem}{Lemma}
\renewcommand{\theThm}{\Alph{Thm}}
\newtheorem*{MLem}{Main Lemma 1}
\newtheorem*{MLemm}{Main Lemma 1'}
\newtheorem*{MCor}{Corollary to the Main Lemma~1}
\newtheorem*{MCoro}{Corollary to the Main Lemma~1'}
\newtheorem*{Pro}{Proposition}
\newtheorem*{DConj}{Descartes conjecture}

\theoremstyle{definition}
\newtheorem*{Def}{Definition}
\newtheorem*{Rem}{Remark}

\begin{document}

\title{\normalsize\bf A new explicit formula in the additive theory of primes with
applications I.  The explicit formula for the Goldbach and
Generalized Twin Prime Problems}

\author{by\\
{J\'anos Pintz\thanks{Supported by ERC-AdG.~321104 and National Research Development and Innovation Office, NKFIH, K~119528.}
}}

\date{}

\numberwithin{equation}{section}
\numberwithin{Lem}{section}

\maketitle

\section{Introduction}
\label{s:1}
The well-known explicit formula of Riemann--Von Mangoldt for the
number of primes up to $x$ ($\varrho = \beta + i \gamma$ denotes
non-trivial zeros of Riemann's zeta-function, $x > 2$, $T \leq
x$),
\begin{equation}
\psi(x) = \sum_{n \leq x} \Lambda(n) = x - \sum_{|\gamma| \leq
T} {x^\rho \over \rho} + O \left({x \over T} \log^2 x\right),
\label{eq:1.1}
\end{equation}
and the analogous ones for $\psi(x, \chi)$ \cite[\S 19]{Dav}, play an
important role in many problems about primes. For example, when
investigating the distribution of primes in short intervals $(x, x
+ y)$, we can subtract the two formulas for $x$ and $x + y$ and
thereby reduce the problem to the density of zeros of $\zeta(s)$.

The aim of the present work is to show that the same approach, that is, to establish an explicit formula
in case of the most famous additive problems about primes
(Goldbach Problem, Generalized Twin Prime Problem), is possible. The explicit
formulas, once established, either lead directly to new results,
or, in other cases, help to reach new results by using other
methods. Another advantage of the explicit formula is that, apart
from the size of the possible exceptional set in Goldbach's
problem, for example, we obtain information about the possible
candidates $n$ for Goldbach-exceptional numbers. (We will call an
even number $n$ a Goldbach number if it can be written as a sum of
two primes, otherwise we will call it a Goldbach-exceptional
number.) The same reasoning is also valid for the previously
mentioned problems. We will now discuss the case of the Goldbach
problem in detail.

Let $E(X)$ denote the number of Goldbach-exceptional numbers up
to $X$. Then Goldbach's conjecture is equivalent to $E(X) = 1$
for $X \geq 2$. Any non-trivial upper estimate for $E(X)$ can be
considered as an approximation to Goldbach's problem. After
Vinogradov \cite{Vin} proved his famous three primes theorem in 1937,
Cudakov \cite{Cud}, Estermann \cite{Est} and Van der Corput \cite{VdC} observed
simultaneously and independently (in 1937--38) that Vinogradov's
method can also yield
\begin{equation}
E(X) \ll X \log^{-A} X \quad \text{ for any } A > 0.
\label{eq:1.2}
\end{equation}

An important step was made by Vaughan \cite{Vau} in 1972 with the
proof of
\begin{equation}
E(X) \ll X \exp (-c \sqrt{\log X}).
\label{eq:1.3}
\end{equation}
Later, in their pioneering work of 1975, Montgomery and Vaughan
\cite{MV} established the estimate
\begin{equation}
E(X) < X^{1 - \delta} \quad \text{ for } X > X_0(\delta),
\label{eq:1.4}
\end{equation}
with a small (theoretically explicitly calculable) $\delta$ and an
effective $X_0(\delta)$.

It turned out to be a very difficult problem to prove
(\ref{eq:1.4}) with some reasonable (not too small) explicit
value of $\delta$ (even with $X_0(\delta)$ ineffective). In 1989
J. R. Chen and J. M. Liu \cite{CL} proved (\ref{eq:1.4}) with $\delta
= 0.05$. This was improved by Hongze Li in 1999 \cite{Li1} to
$\delta = 0.079$, and in 2000 \cite{Li2} to
\begin{equation}
E(X) < X^{0.914} \ \text{ for } X > X_1, \text{an
ineffective constant.}
\label{eq:1.5}
\end{equation}

This was improved further by Wen Chao Lu \cite{Lu} in 2010 to
\beq
E(X) < X^{0.879} \ \text{ for } \ X > X_2, \ \text{ an ineffective constant.}
\eeq

In order to illustrate the differences in the methods of proof
of (\ref{eq:1.2}) and (\ref{eq:1.4}), we define
\begin{equation}
S(\alpha) = \sum_{X_1 < p \leq X} \log pe(\alpha p),\
e(u) = e^{2\pi iu}, \ X_1 = X^{1-\varepsilon_0},\
\mathcal L = \log X
\label{eq:1.6}
\end{equation}
with $\varepsilon_0$, an arbitrary small positive constant.

To dissect the unit interval, we will choose a $P$ with
\begin{equation}
\mathcal L^c \leq P \leq \sqrt{X}, \quad Q = X/P, \quad \vartheta = \frac{\log P}{\log X}
\label{eq:1.7}
\end{equation}
and define the major arcs $\mathfrak M$ as the union of the
non-overlapping arcs $\mathfrak M(q, a) = [a/q - 1/q Q, a/q
+ 1/qQ]$ for $q \leq P$. Let
\begin{equation}
\mathfrak M = \bigcup_{q \leq P} \bigcup_{\substack{ a\\ (a,q) = 1}}
\mathfrak M(q,a),
\label{eq:1.8}
\end{equation}
and denote the minor arcs by ${\mathfrak m} = [1/Q, 1 + 1/Q]
\setminus \mathfrak M$. Then for any even $m \in [\mathcal L X_1,
X]$ we can write
\begin{equation}
R(m) = \sum_{\substack{p + p' = m\\ p, p' > X_1}} \log p \cdot \log
p' = R_1(m) + R_2(m),
\label{eq:1.9}
\end{equation}
where
\begin{equation}
R_1(m) = \int\limits_{\mathfrak M} S^2(\alpha) e(-m\alpha) d\alpha,\quad
R_2(m) = \int\limits_{\mathfrak m} S^2(\alpha) e(-m\alpha) d\alpha.
\label{eq:1.10}
\end{equation}

We will suppose $m \in [X/2, X]$ for convenience.
In general, in the circle method $P$ is chosen to be as large as
possible, with the condition that the contribution $R_1(m)$ can
be evaluated asymptotically, yielding the expected main term
\begin{equation}
R_1(m) \sim {\mathfrak S}(m) \cdot I(m), \quad
I(m) = \sum_{\substack{k + \ell = m\\ k,\ell \in [X_1,X]}}  1 = m - 2
X_1 + O(1),
\label{eq:1.11}
\end{equation}
where
\begin{equation}
{\mathfrak S}(m) = \prod_{p|m} \left(1 + {1 \over p - 1}\right)
\prod_{p \nmid m} \left(1 - {1 \over (p-1)^2}\right).
\label{eq:1.12}
\end{equation}
In order to show (\ref{eq:1.11}), we usually require that primes
should be uniformly distributed in all arithmetic progressions
modulo $q$ for all $q \leq P$. Such a result, the famous
Siegel--Walfisz theorem (established in 1936), played a crucial
role in the proof of (1.2), and in the Goldbach--Vinogradov
theorem as well. By this theorem one can choose $P = \mathcal
L^A$ ($A$ arbitrary large constant).
After this, Vinogradov's famous estimate for $S(\alpha)$ on the
minor arcs (see Lemma \ref{l:4.10}), combined with Parseval's
identity leads to the fact that $R_2(m) = o({\mathfrak S}(m)m)$
for all but $\mathcal L^C X/P$ even integers $m \leq X$ (see
Section~5).

Montgomery--Vaughan's ingenious idea is to choose a larger value,
$P = X^\delta$. In this case possible zeros of Dirichlet $L$-functions near to the line
$\sigma = 1$ may destroy the uniform distribution of primes with
respect to moduli less than $P$. If there is no Siegel zero (see
(\ref{eq:4.13})--(\ref{eq:4.14})), then we have a statistically
good distribution of primes in arithmetic progressions, the famous
Gallagher prime number theorem \cite[Theorem 6]{Gal}.
This substitutes for the uniform distribution of primes in all
arithmetic progressions, therefore we may prove the (still
sufficient) inequality
\begin{equation}
R_1(m) \gg {\mathfrak S}(m)m
\label{eq:1.13}
\end{equation}
in place of (\ref{eq:1.11}).

If there is a Siegel zero, this might completely destroy the
picture. This can be seen very easily, without the circle method,
in the following way. Suppose, for simplicity, that we have a
character $\chi_1 \mod  q$, where $\chi_1(-1) = -1$,
and $L(1 - \delta_1, \chi_1) = 0$ for a very small $\delta$.
Let us consider $R(m)$ (see \eqref{eq:1.9}) for $q|m$. If $p + p' = m$, $p \nmid q$,
then $\chi_1(p) = 1$ or $\chi_1(p') = 1$, and so
\begin{equation}
R_1(m) \ll \log m \sum_{\substack{p\leq m\\ \chi_1(p)=1}} \log p \ll
\log m \left( m - {m^{1 - \delta_1} \over 1 - \delta_1}\right)
\ll \delta_1 m \log^2 m,
\label{eq:1.14}
\end{equation}
which might be very small, since we can assume only $\delta_1
\gg m^{-\varepsilon}$.

Thus in case of the existence of a Siegel zero, Montgomery and
Vaughan evaluate exactly the effect of the Siegel zero for
$R_1(m)$, and they obtain for it an additional term
\begin{equation}
\widetilde{\mathfrak S}(m)\widetilde I(m),
\label{eq:1.15}
\end{equation}
which may almost cancel the effect of the main term ${\mathfrak S}(m)m$ for many values of $m$ (for example, for
the multiples of $q$).
But the cancellation cannot be complete, since \cite[\S 6]{MV}
\begin{equation}
|\widetilde{\mathfrak S}(m)|\leq {\mathfrak S}(m) \quad \text{(with equality possible)}
\label{eq:1.16}
\end{equation}
and
\begin{equation}
\widetilde I(m) = \sum_{X_1 < k < X - X_1} (k(m -
k))^{-\delta_1} \leq I(m) - c\delta_1 m \log m.
\label{eq:1.17}
\end{equation}
Now, in the case of the existence of a Siegel zero, other $L$-functions are free from zeros near $\sigma = 1$
by the Deuring--Heilbronn phenomenon (see Lemma \ref{l:4.21}).
Therefore, one can prove the still-sufficient inequality
\begin{equation}
R_1(m) \geq (1 + o(1)) {\mathfrak S}(m)(I(m) - \widetilde I(m)) \gg
\delta_1 {\mathfrak S}(m) m \log m.
\label{eq:1.18}
\end{equation}

Our method is a generalization of the Montgomery--Vaughan
method. We will choose a $P$ less than $X^{4/9-\eta}$, $\eta > 0$
arbitrary.
We will introduce singular series ${\mathfrak S}(\chi_1, \chi_2, m)$
for every pair of primitive characters $\chi_1, \chi_2$ modulo $r_1,
r_2$ with $[r_1, r_2] \leq P$. (We consider the trivial
character $\chi_0(n) = 1$ as a primitive character $\mod 1$.)
We can evaluate these singular series and show an explicit
formula for it, which implies
\begin{equation}
\bigl|{\mathfrak S}(\chi_1, \chi_2, m)\bigr| \leq {\mathfrak S}(m),
\label{eq:1.20}
\end{equation}
and further
\begin{equation}
\bigl|{\mathfrak S}(\chi_1, \chi_2, m)\bigr| \leq {{\mathfrak S}(m) \over \sqrt{U}} \log^2_2 U,
\label{eq:1.21}
\end{equation}
where
\begin{equation}
U\! =\! U(\chi_1, \chi_2, m)\! =\! \max \left(\!
{r^2_1 \over (r_1, r_2)^2} , {r^2_2 \over (r_1, r_2)^2}, {r_1
\over (|m|, r_1)}, {r_2 \over (|m|, r_2)},\, \mathrm{cond}\, \chi_1
\chi_2\! \right)\!.
\label{eq:1.22}
\end{equation}
This is proved in our Main Lemma in Section~7.
Further, it is shown there that the sum of the absolute
values of the elements in the
singular series of ${\mathfrak S}(\chi_1, \chi_2, m)$ will be $\leq
c|{\mathfrak S}(\chi_1, \chi_2, m)|$ (not just $\leq c{\mathfrak S}(m)$, as in Lemma~5.5 of \cite{MV}).

In the same way as for $\widetilde I(m)$, one can evaluate the
effect of any pair of zeros:
\begin{equation}
I(\varrho_1, \varrho_2, m) \stackrel{\mathrm{def}}{=}
\sum_{\substack{m=k+\ell\\ X_1 < k,\ell \leq X}} k^{\varrho_1 - 1}
\ell^{\varrho_2 - 1} = {\Gamma(\varrho_1) \Gamma(\varrho_2)
\over \Gamma(\varrho_1 + \varrho_2)} m^{\varrho_1 + \varrho_2 -
1} + O(X_1),
\label{eq:1.23}
\end{equation}
when $|\gamma_i| \leq X^{1 - \varepsilon_0}$, for example (see
Lemma \ref{l:4.9}).

In such a way we will obtain both the main term ${\mathfrak S}(m)I(m)$
and a uniformly bounded number of ``supplementary main terms''
which have the form
\begin{equation}
{\mathfrak S}(\chi_1, \chi_2, m) I(\varrho_1, \varrho_2, m)
\label{eq:1.24}
\end{equation}
with a bounded number of possible generalized exceptional zeros $\varrho_\nu$
belonging to $L(s, \chi_\nu)$ with
$\chi_\nu$, $\nu = 1,2,\dots K$, $0 \leq K \leq K_0$,
\begin{equation}
\varrho_\nu = 1 - \delta_\nu + i\gamma_\nu, \quad
\delta_\nu \leq H / \mathcal L, \quad |\gamma_\nu| \leq U,
\label{eq:1.25}
\end{equation}
where $H, U$ are large constants and $K_0 = K_0(H, U)$.

Using the convention that the pole $\varrho_0 = 1$ of $L(s, \chi_0)$ is included
with the possibly existing zeros, with the notation
\begin{equation}
A(\varrho) = 1 \ \text{ if } \ \varrho = \varrho_0 = 1, \ \chi =
\chi_0 (\mod 1)
\label{eq:1.26}
\end{equation}
\begin{equation}
A(\varrho_\nu) = -1 \ \text{ if } \  L(\varrho_\nu,
\chi_\nu) = 0 \quad
(\nu = 1,2,\dots K),
\label{eq:1.27}
\end{equation}
we obtain the explicit formula for the contribution of the major arcs:
\begin{align}
R_1(m) &= \sum^{K+1}_{\nu=0} \sum^{K+1}_{\mu=0} A(\varrho_\nu)
A(\varrho_\mu) {\mathfrak S}(\chi_\nu, \chi_\mu, m) I (\varrho_\nu, \varrho_\mu,
m)\label{eq:1.28}\\
&\quad + O(Xe^{-cH}) + O(XU^{-1/2}). \nonumber
\end{align}

This formula and the above mentioned information (cf.\ \eqref{eq:1.20}--\eqref{eq:1.22}) about the properties of the generalized singular series $\mathfrak S(\chi_\nu, \chi_\mu, m)$, together with its
analogue for the Generalized Twin Prime Problem, will have a
number of arithmetic consequences, to be proven in later works.
For example, we will show in later parts of this series the following

\begin{Thm}
\label{th:A}
$\int\limits_{\mathfrak M} |S(\alpha)|^2 e(-m\alpha)d\alpha = (1
+ o(1)) {\mathfrak S}(m)X$, if $m$ is fixed, $X \to \infty$.
\end{Thm}

\begin{Thm}
\label{th:B}
All but $O(X^{3/5} \log^{10}X)$ odd numbers can be written as
the sum of three primes with one prime less than $C$, a given
absolute constant.
\end{Thm}

We can show about the gaps between consecutive Goldbach numbers

\begin{Thm}
\label{th:C}
\[\sum_{g_n \leq x} (g_{n+1} - g_n)^\gamma = 2^{\gamma - 1} X +
O(X^{1 - \delta}) \ \text{ for } \gamma < {341 \over 21},
\]
where $g_n$ is the $n$-th Goldbach number.
\end{Thm}

We remark that Mikawa \cite{Mik} proved the above but just for $\gamma < 3$.

Descartes (1596--1650) expressed a conjecture similar to Goldbach's one already in the 17\textsuperscript{th} century, which however appeared in a printed format as late as in 1908 \cite{Des}.

\begin{DConj}
Every even integer can be expressed as a sum of at most three primes.
\end{DConj}

Since in this case one of the summands has to be two, at the first sight we might think this is equivalent to the Goldbach conjecture.
However, it is in fact equivalent to the assertion that for every even $N$ at least one of $N$ or $N + 2$ is a Goldbach number (i.e.\ the sum of two primes).
Our new methods are able to handle such type of problems more efficiently than Goldbach's problem (in contrast to earlier methods).

We can show for example that our present results imply

\begin{Thm}
\label{th:D}
For every $\varepsilon > 0$, all but $O_\varepsilon
(X^{3/5+\varepsilon})$ positive integers $m \leq X$ can be
written as a sum of at most three primes or prime-powers.
\end{Thm}

Theorem~\ref{th:D} will be an easy consequence of

\begin{Thm}\label{th:E}
There are explicitly calculable absolute constants $K$ and $C_3$ such that for all but $C_3 X^{3/5} \log^{12} X$ numbers 
$n \leq X$ we have 
\beq
E(n + \log^2 n) - E(n) \leq K.
\eeq
\end{Thm}

The following results will also be based on the explicit formula,
but their proof will require still many further ideas.

\begin{Thm}\label{th:ujF} {\rm (J. Pintz\,--\,I. Ruzsa).}
Every sufficiently large even integer can be written as the sum of two primes and eight powers of two.
\end{Thm}

The best published unconditional result is due to Heath-Brown and Puchta \cite{HP} with $13$ powers of two.

%%%itt~tartok180311
\begin{Thm}
\label{th:F}
For every $\varepsilon > 0$, all but $O_\varepsilon
(X^{3/5+\varepsilon})$ positive integers $m \leq X$ can be
written as a sum of at most three primes.
\end{Thm}

\begin{Thm}
\label{th:H}
$E(X) < X^{3/4}$ for $X > C$.
\end{Thm}

\section{Statement of results}
\label{s:2}
In order to formulate the explicit formula we need some more
notation. For any $\chi \mod q$ let
\begin{equation}
c_\chi(m) = \sum^q_{h=1} \chi(h)e\left({hm\over q}\right), \quad
\tau(\chi) = c_\chi(1).
\label{eq:2.1}
\end{equation}
Further for primitive characters $\chi_i \mod r_i$  ($r_i = 1$ is
possible), $r_i\mid q$ $(i=1,2)$ let
\begin{equation}
c(\chi_1, \chi_2, q, m) = \varphi^{-2}(q) c_{\chi_1 \chi_2 \chi_{0,q}} (-m)
\tau (\overline \chi_1 \chi_{0,q}) \tau (\overline \chi_2 \chi_{0,q}),
\label{eq:2.2}
\end{equation}
\begin{equation}
{\mathfrak S}(\chi_1, \chi_2, m) = \sum^\infty_{\substack{q = 1\\ [r_1, r_2]\mid q}}
c(\chi_1, \chi_2, q, m),
\label{eq:2.3}
\end{equation}
where $\chi_{0,q}$ is the principal character $\mod q$.
Let $\mathrm{cond}\, \chi$ denote the conductor of a character~$\chi$.

In case of the Generalized Twin-Prime Problem we need
\begin{equation}
c'(\chi_1, \chi_2, q, m) = \varphi^{-2}(q) c_{\chi_1 \overline \chi_2 \chi_{0,q}} (-m)
\tau (\overline \chi_1 \chi_{0,q}) \overline{\tau (\overline \chi_2 \chi_{0,q})},
\label{eq:2.4}
\end{equation}
\begin{equation}
{\mathfrak S}'(\chi_1, \chi_2, m) = \sum^\infty_{\substack{q = 1\\ [r_1, r_2]\mid q}}
c'(\chi_1, \chi_2, q, m),
\label{eq:2.5}
\end{equation}
\begin{equation}
R_1(m) = \int\limits_{\mathfrak M} S^2(\alpha) e(-m\alpha) d\alpha,
\quad
R'_1(m) = \int\limits_{\mathfrak M} |S^2(\alpha)| e(-m\alpha)d\alpha,
\label{eq:2.6}
\end{equation}
\begin{equation}
I(\varrho_1, \varrho_2, m) = \sum_{\substack{m=k+\ell\\ k,\ell\in(X_1, X]}}
 k^{\varrho_1-1} \ell^{\varrho_2-1},
\quad
I'(\varrho_1, \varrho_2, m) = \sum_{\substack{m=k - \ell\\ k,\ell\in (X_1, X]}}
 k^{\varrho_1-1} \ell^{\overline{\varrho_2}-1}.
\label{eq:2.7}
\end{equation}

Let us define the set $\mathcal E = \mathcal E(H,P,X)$ of
generalized exceptional singularities of the functions $L'/L$ for all primitive
$L$-functions $\mod r$, $r \leq P$, as follows $(\chi_0 = \chi_0(\mod
1)$ corresponds to $\zeta(s)$).
\begin{equation}
\aligned
(\varrho_0, \chi_0) &\in \mathcal E \ \text{ with } \varrho_0 = 1\\
(\varrho_\nu, \chi_\nu) &\in \mathcal E \ \text{ if }
\exists \chi_\nu \ (\nu \geq 0), \,\mathrm{cond}\, \chi_\nu = r_\nu \leq P, \quad
 L(\varrho_\nu, \chi_\nu) = 0, \\
\beta_\nu &\geq 1 - H/\mathcal L, |\gamma_\nu| \leq \sqrt{X},
\endaligned
\label{eq:2.8}
\end{equation}
where $H$ will be a sufficiently large constant to be chosen later.
We remark that the best known zero-free regions for $\zeta(s)$ exclude the possibility that $\zeta(s)$ would have additional exceptional singularities beyond $\varrho_0 = 1$ for sufficiently large values of~$X$.

Further let
\[
\mathcal E_T = \{ \varrho \in \mathcal E; \mid |\mathrm{Im}\varrho| \leq T\}.
\]
Let us consider a $P_0 \leq X^{4/9 - \eta_0}$ where $\eta_0$ is
any positive number. Every further constant or parameter, as
well as $\varepsilon_0$ in the definition of $X_1$ in (1.6) may
depend on $\eta_0$. We suppose that $X$ exceeds some effective constant
$X_0(\eta_0)$.

We can fix a sufficiently small $h = h_0$ (depending also on
$\eta_0$, and $c_1$ in (\ref{eq:4.14})) and introduce the

\begin{Def}
We call $\varrho_1 = 1 - \delta_1$, a real zero of $L(s, \chi_1)$ with a real character
$\chi_1$, a Siegel zero (with respect to $h$, $P$ and $X$) if
\begin{equation}
\delta_1 \leq h / \mathcal L, \quad \text{cond}\, \chi_1 \leq P.
\label{eq:2.9}
\end{equation}
\end{Def}

\begin{Rem}
If we have chosen $h = h_0$ small enough, then in view of
Lemma~\ref{l:4.13} we have at most one, simple Siegel zero
belonging to one primitive character $(h_0 \leq c_1 \mathcal
L/\log P)$.
\end{Rem}

\renewcommand{\theThm}{\arabic{Thm}}
\setcounter{Thm}{0}

With the notation of {\rm (\ref{eq:1.6})--(\ref{eq:1.8})},
{\rm (\ref{eq:1.26})--(\ref{eq:1.27})} and \eqref{eq:2.1}--\eqref{eq:2.9} we have

\begin{Thm}
\label{t:1}
For every $P_0 \leq X^{4/9 - \varepsilon}$ we can choose a $P
\in [P_0 X^{-\varepsilon}, P_0]$ with the following properties.
We have for all $m\leq X$ the explicit formulas
\begin{equation}
\aligned
R_1(m) &= \sum_{(\varrho_i, \chi_i) \in \mathcal E}
\sum_{(\varrho_j, \chi_j) \in
\mathcal E}
A(\varrho_i) A(\varrho_j) {\mathfrak S}(\chi_i, \chi_j, m) {\Gamma(\varrho_i)
\Gamma(\varrho_j) \over \Gamma(\varrho_i + \varrho_j)}
m^{\varrho_i  + \varrho_j - 1}\\
& \quad + O_\varepsilon({\mathfrak S}(m)Xe^{-c_\varepsilon H}) + O_\varepsilon (X^{1 -
\varepsilon_0}),
\endaligned
\label{eq:2.10}
\end{equation}
\begin{align}
R_1'(m) &= \sum_{(\varrho_i, \chi_i) \in \mathcal E}
\sum_{(\varrho_j,  \chi_j) \in
\mathcal E}
A(\varrho_i) A(\varrho_j) {\mathfrak S}'(\chi_i, \chi_j, m) I'(\varrho_i,
\varrho_j, m)\label{eq:2.11}\\
&\quad + O_\varepsilon ({\mathfrak S}(m) Xe^{-c_\varepsilon H}) + O_\varepsilon (X^{1-\varepsilon_0}).
\nonumber
\end{align}
Suppose additionally $m \in [X/4, X/2]$. Then, replacing the
summation condition {\rm (\ref{eq:2.10})--(\ref{eq:2.11})} by
\begin{equation}
\underset{\hspace*{5mm}[r_1, r_2] \leq P,\quad U(\chi_1, \chi_2, m) \leq U}{
\sum_{\substack{(\varrho_i, \chi_i ) \in \mathcal E\\
|\gamma_i| \leq U}} \sum_{\substack{(\varrho_j, \chi_j) \in \mathcal E\\
|\gamma_j| \leq U}}}
\label{eq:2.12}
\end{equation}
(in case of {\rm (\ref{eq:2.11})} $U(\chi_1, \chi_2, m)$ should be replaced
by $U(\chi_1, \overline \chi_2, m)$), we obtain
{\rm (\ref{eq:2.10})--(\ref{eq:2.11})} with an additional error term
\[
O({\mathfrak S}(m) X \log U/\sqrt{U}).
\]
\end{Thm}

Formulae (\ref{eq:2.10}) and (\ref{eq:2.11}) are quite
satisfactory with respect to the error terms if there is no
Siegel zero (in this case one can choose $H$ and $U$ large
constants).
However, this is not the case if we have a Siegel zero.

The following theorem overcomes this difficulty.

Further, in case of \eqref{eq:2.13} we have for all but $O(X^{3/5 + \ve} + \varepsilon)$ values of $m \in [X/2, X]$: \
$R_1(m) \gg_{\varepsilon} m^{1 - \varepsilon}$, $R_1'(m)\gg_\varepsilon m^{1 - \varepsilon}$.

\begin{Thm}
\label{t:2}
Let $\varepsilon > 0$ be arbitrary. If $X > X(\varepsilon)$,
ineffective constant, and there exists a Siegel zero $\beta_1$
of $L(s, \chi_1)$ with
\begin{equation}
\beta_1 > 1 - h / \log X, \quad \mathrm{cond}\, \chi_1 \leq X^{4/9
- \varepsilon},
\label{eq:2.13}
\end{equation}
where $h$ is a sufficiently small constant, depending on
$\varepsilon$, then
\begin{equation}
E(X) < X^{3/5 + \varepsilon}
\label{eq:2.14}
\end{equation}
and, similarly
\begin{equation}
E'(X) = \big|\{m \leq X;\ 2 \mid m,\ m \neq p - p'\} \big| <
X^{3/5+\varepsilon} .
\label{eq:2.15}
\end{equation}
\end{Thm}

In view of the zero-free region for $L$-functions in
Lemma~\ref{l:4.12}, Theorems~\ref{t:1} and \ref{t:2} immediately
imply

\begin{Thm}
\label{t:3}
There are explicitly calculable positive constants $C_1, c_2,
C_3$ with the following property. If $L(s, \chi) \neq 0$
for
\begin{equation}
1 - {C_1 \over \log q} \leq \sigma \leq 1 - {c_2 \over \log q},
\quad |t| \leq C_3,
\label{eq:2.16}
\end{equation}
then the estimates {\rm (\ref{eq:2.14})--(\ref{eq:2.15})} hold
for every $\varepsilon > 0$ in case of $X > X'(\varepsilon)$.
\end{Thm}

The reason for the implication is the following.
If there exists a zero with $\sigma \geq 1 - c_2/\log q$, $|t| \leq C_3$, $q \leq X^{4/9}$, then by Lemma~\ref{l:4.12} this has to be a Siegel zero.
Consequently, \eqref{eq:2.15} follows from Theorem~\ref{t:2}.
If, on the other hand, the whole range $1 - C_1 / \log q \leq \sigma \leq 1$, $|t| \leq c_3$, $q \leq X^{4/9}$ is zero-free, then the crucial sums in \eqref{eq:2.10}--\eqref{eq:2.12} contain only the main term if the constants $C_1 = H$, $C_3 = U$ were chosen sufficiently large.

In comparison we note that under the assumption of the
Generalized Riemann Hypothesis (in place of the much weaker
condition (\ref{eq:2.16})) Hardy--Littlewood \cite{HL} proved in 1924
the estimate $E(X) \ll X^{1/2 + \varepsilon}$.

We remark further that one can show that Theorems~\ref{t:1} and \ref{t:2} also
imply Montgomery--Vaughan's estimate (\ref{eq:1.4}).

\section{Notation}
\label{s:3}

Beyond the notation of Sections~{1} and \ref{s:2} (cf.\
(\ref{eq:1.6})--(\ref{eq:1.12}), (\ref{eq:1.17}),
(\ref{eq:1.22}), (\ref{eq:1.23}), (\ref{eq:1.25}),
(\ref{eq:1.26})--(\ref{eq:1.27}), (\ref{eq:2.6}), (\ref{eq:2.8}),
(\ref{eq:2.15})) we will use the following notation. The symbol
$\varrho = \varrho_\chi$ will denote a zero or a pole of
$L(s,\chi)$, where $\chi$ will denote mostly primitive
characters. Let
\begin{equation}
\varrho = \beta + i \gamma = 1 - \delta + i\gamma,
\label{eq:3.1}
\end{equation}
\begin{equation}
N(\alpha, T, \chi) = \sum_{\substack{\varrho = \varrho_\chi \\ \beta \geq
\alpha, |\gamma| \leq T}} 1
\label{eq:3.2}
\end{equation}
\begin{equation}
N^*(\alpha, T, Q) = \sum_{q \leq Q} \underset{\chi(q)}{\sum\nolimits^*}
N(\alpha, T, \chi),
\label{eq:3.3}
\end{equation}
where $\underset{\chi(q)}{\sum\nolimits^*}$ means a summation over
primitive characters $\mod q$. Further
$\underset{a(q)}{\sum\nolimits^\prime}$ will denote summation
over all reduced residue classes. Let
\begin{equation}
T(\varrho, \eta) = \sum_{X_1 < n \leq X} n^{\varrho-1} e(n\eta).
\label{eq:3.4}
\end{equation}
Further, $r \sim R$ will denote $R \leq r < 2R$.

\section{Auxiliary results}
\label{s:4}

The following arithmetic results appear as Lemmas 5.1--5.4 of \cite{MV}.

\begin{Lem}
\label{l:4.1}
If $\chi$ is a primitive character $(\mod q)$ then $|\tau(\chi)| =
q^{1/2}$.
\end{Lem}

\begin{Lem}
\label{l:4.2}
Let $\chi$ be a character $(\mod k)$, induced by a primitive
character $\chi^*(\mod r)$. Then $r\mid k$ and
\begin{equation}
\tau(\chi) = \mu \left({k \over r} \right) \chi^* \left({k\over
r}\right) \tau(\chi^*).
\label{eq:4.1}
\end{equation}
\end{Lem}

\begin{Lem}
\label{l:4.3}
Suppose the above hypotheses hold, and that $(m,k) = 1$. Then
\begin{equation}
c_\chi(m) = \overline{\chi^*}(m) \mu \left({k \over r} \right) \chi^* \left({k\over
r}\right) \tau(\chi^*).
\label{eq:4.2}
\end{equation}
\end{Lem}

\begin{Lem}
\label{l:4.4}
Let $\chi$ be a character $(\mod q)$, induced by a primitive
character $\chi^*(\mod r)$. For an arbitrary integer $m$ put $q_1 =
q / (q, |m|)$. If $r \nmid q_1$ then $c_\chi(m) = 0$. If
$r \mid q_1$ then
\begin{equation}
c_\chi (m) = \chi^*\left({m\over (q, |m|)} \right) {\varphi(q) \over
\varphi(q_1)} \mu \left({q_1 \over r} \right) \chi^* \left({q_1
\over r} \right) \tau (\chi^*).
\label{eq:4.3}
\end{equation}
\end{Lem}

We will use the following (mostly) well-known results from the
theory of exponential sums

\begin{Lem}
\label{l:4.5}
Let $F(x)$ be a real differentiable function such that $F'(x)$
is monotonic and $F'(x) \geq m > 0$, or $F'(x) \leq - m < 0$, in
$(a,b)$. Then
\begin{equation}
\bigg| \int\limits^b_a e^{iF(x)} dx \bigg| \leq {4 \over m}.
\label{eq:4.4}
\end{equation}
\end{Lem}

This is Lemma 4.2 of Titchmarsh \cite{Tit}.

\begin{Lem}
\label{l:4.6}
Let $f(x)$ be a real differentiable function in $(a,b)$, $f'(x)$
monotonic, $|f'(x)| \leq \theta < 1$. Then
\begin{equation}
\sum_{a < n\leq b} e(f(n)) = \int\limits^b_a e(f(x)) dx + O(1).
\label{eq:4.5}
\end{equation}
\end{Lem}

This is Lemma 4.8 of Titchmarsh \cite{Tit}.

\begin{Lem}
\label{l:4.7}
Let $0 \leq \sigma \leq 1$, $|t|\leq x$. Then we have uniformly
\begin{equation}
\sum_{x < n \leq N} n^{-s} = \int\limits^N_x u^{-s} du + O(x^{-\sigma}),
\label{eq:4.6}
\end{equation}
with an absolute constant (independent of $s$ too) implied by
the $O$ symbol.
\end{Lem}

\begin{proof}
This relation is contained in the proof of Theorem~4.11 of \cite{Tit}.
However, for this part we may allow $0 \leq \sigma \leq 1$,
since the proof follows from Lemma~4.10 of \cite{Tit}.
\end{proof}

\begin{Lem}
\label{l:4.8}
The Euler beta function $B(u,v)$, defined below for
$\mathrm{Re}\, s > 0$, $\mathrm{Re}\, w > 0$ satisfies the equation
\begin{equation}
B(s, w) \overset{\mathrm{def}}{=} \int\limits^1_0 x^{s-1} (1 -
x)^{w-1} dx = {\Gamma(s) \Gamma(w)\over \Gamma(s+w)}.
\label{eq:4.7}
\end{equation}
\end{Lem}

This can be found e.g.\ in Chapter 3 of \cite{Kar}.

The following lemma may be well known, but we did not find any
exact references:

\begin{Lem}
\label{l:4.9}
Let $s = \sigma + it$, $w = \lambda + iv$, $0 < \sigma$, $\lambda \leq 1$, $Y \geq 1$, $\max(|t|, |v|) \leq Y$.
Then we have for any integer $m \geq 2Y$
\begin{equation}
\sum_{Y<k\leq m - Y} k^{s-1}(m - k)^{w-1}
= {\Gamma(s)\Gamma(w)
\over \Gamma(s + w)} m^{s+w-1}  + O(Y).
\label{eq:4.8}
\end{equation}
\end{Lem}

\begin{proof}
Let us suppose by symmetry $|w| \leq |s|$ and denote
\begin{equation}
K(x) = \sum_{Y < k \leq x} k^{s-1}, \quad
J(x) = \int\limits^x_Y y^{s-1}dy.
\label{eq:4.9}
\end{equation}
Then by partial summation and integration, resp., we obtain by
(\ref{eq:4.6})--(\ref{eq:4.7}) for the sum $S$ in (\ref{eq:4.8})
\begin{align}
S &= K(m - Y) Y^{w-1} - \int\limits^{m-Y}_Y K(u)((m - u)^{w -
1})' du \label{eq:4.10}\\
&= J(m - Y) Y^{w - 1} - \int\limits^{m - Y}_Y J(u)((m - u)^{w -
1})' du + O(1) \nonumber\\
&= \int\limits^{m-Y}_Y J'(u)(m - u)^{w-1} du + O(1)  \nonumber \\
&= \int\limits^m_0 u^{s-1} (m - u)^{w-1} du + O(Y) \nonumber\\
&= {\Gamma(s)\Gamma(w) \over \Gamma(s + w)} m^{s+w-1} + O(Y). \qedhere
 \nonumber
\end{align}
\end{proof}

Vinogradov's famous estimate on the minor arcs
was substantially simplified by Vaughan (for the proof see \cite[Chapter 25]{Dav}).

\begin{Lem}
\label{l:4.10}
For $|\alpha - a/q| \leq q^{-2}$, $(a,q) = 1$ we have
\begin{equation}
\sum_{p \leq N} \log p e(p\alpha) \ll (Nq^{-1/2} + N^{4/5} +
(Nq)^{1/2}) \log^4 N.
\label{eq:4.11}
\end{equation}
\end{Lem}

The following lemma of Gallagher \cite[Lemma 1]{Gal} makes possible
the estimation of integrals for $|S^2_i(\alpha)|$ (see
(\ref{eq:6.2})--(\ref{eq:6.3}) via density theorems for zeros of
$L$-functions).

\begin{Lem}
\label{l:4.11}
Let $u_1, u_2, \dots, u_N$ be arbitrary real numbers. Then for
any $\kappa > 0$
\begin{equation}
\int\limits^\kappa_{-\kappa} \bigg| \sum u_n e(n\eta)\bigg|^2
d\eta \ll
\int\limits^\infty_{-\infty} \bigg| \kappa
\sum^{x+(2\kappa)^{-1}}_x u_n \bigg|^2 dx.
\label{eq:4.12}
\end{equation}
\end{Lem}

The zero-free region for $L$-functions can be given by
the following

\begin{Lem}
\label{l:4.12}
Let $q \geq 1$ be any integer. There exists an absolute constant
$c_0$ such that
\begin{equation}
L(s, \chi) \neq 0 \ \text{ for } \sigma > 1 - {c_0 \over
\max(\log q, \log^{3/4}(|t|+2))}
\label{eq:4.13}
\end{equation}
with the possible exception of at most one, simple real zero $\beta_1$
of an $L$-function corresponding to a real exceptional character
$\chi_1 \mod q$.
\end{Lem}

This is Satz 6.2 of Chapter VIII in \cite{Pra}; the possibly existing
exceptional zeros are often called Siegel zeros.

The following result is a reformulation of a theorem of Landau
(for a proof see \cite[\S 14]{Dav}).

\begin{Lem}
\label{l:4.13}
There is a constant $c_1 > 0$ such that
there is at most one real primitive $\chi$ to a modulus $\leq z$
for which $L(s, \chi)$ has a real zero $\beta$ satisfying
\begin{equation}
\beta > 1 - {c_1 \over \log z}.
\label{eq:4.14}
\end{equation}
We remark that for $z$ large enough, $c_1 = \frac12 + o(1)$ can be
chosen\/ {\rm \cite{Pin}}.
\end{Lem}

Siegel's theorem (\cite[\S 14]{Dav}) gives an upper estimate for $\beta$:

\begin{Lem}
\label{l:4.14}
For any $\varepsilon > 0$ there exists a positive ineffective
constant $c(\varepsilon)$ such that if $\chi$ is a real character
$\mod q$, $L(\beta, \chi) = 0$, $\beta$ real, then
\begin{equation}
\beta < 1 - c(\varepsilon)q^{-\varepsilon}.
\label{eq:4.15}
\end{equation}
\end{Lem}

We will use the explicit formula for $\psi(x, \chi)$ in the
following form.

\begin{Lem}
\label{l:4.15}
Let $\chi$ be any character $\mod q$, $T \geq \sqrt{x}$, $x \geq
2$. Let $E(\chi) = 1$ if $\chi = \chi_0$, $E(\chi) = 0$ otherwise. Then we have
\begin{equation}
\psi(x, \chi) \overset{\text{\rm def}}{=} \sum_{p \leq x} \chi(p) \log p = E(\chi )x - \sum_{\substack{|\gamma| \leq T\\
\beta \geq 1/2}}
{x^\varrho \over \varrho} + O(\sqrt x \log^2 qx).
\label{eq:4.16}
\end{equation}
\end{Lem}

\begin{proof}
It follows from formulas (7)--(8) of \S 19 of \cite{Dav}, after a
trivial estimate for the contribution of prime-powers to
$\psi(x, \chi)$.
\end{proof}

The following zero-density estimates for $L$-functions
will be used in the sequel. (In the following $Q \geq 1$, $T
\geq 2$, $1/2 \leq \alpha \leq 1$, $\varepsilon > 0$ is an
arbitrary positive number.)

\begin{Lem}
\label{l:4.15a}
$N^*(\alpha, T, Q) \ll (Q^2 T)^{3(1-\alpha) \over 2-\alpha} \log^9 QT$.
\end{Lem}

This is Theorem 12.2 of Montgomery \cite{Mon}.

\begin{Lem}
\label{l:4.16}
$N^*(\alpha, T, Q) \ll_\varepsilon (Q^2
T^{6/5})^{{20\over 9}(1-\alpha)+\ve}$.
\end{Lem}

This is Theorem 2 of Heath--Brown \cite{Hea}.

\begin{Lem}
\label{l:4.17}
$N^*(\alpha, T, Q) \ll_\ve(Q^2 T)^{(2+\ve)(1-\alpha)}$ for
$\alpha \geq 4/5$.
\end{Lem}

This is Theorem 1 of Jutila \cite{Jut}.

Lemmas \ref{l:4.16} and \ref{l:4.17} clearly imply for $1/2 \leq
\alpha \leq 1$

\begin{Lem}
\label{l:4.18}
$N^*(\alpha, T, Q) \ll (Q^2 T^{6/5})^{({20\over 9} +
\varepsilon)(1-\alpha)}$.
\end{Lem}

The following two ``log-free'' density theorems were proved \cite[Corollary 1 and Theorem 2]{Pin}.

\begin{Lem}
\label{l:4.19}
For $h < 1/5$ we have
\begin{equation}
N^*(1-h, T, Q) \ll_\varepsilon \left(Q^{(3+\varepsilon)(3-4h)
\over 4(1-4h)(1-2h)} T^{3+\varepsilon \over 2(1 - 4h)}\right)^h.
\label{eq:4.17}
\end{equation}
\end{Lem}

\begin{Lem}
\label{l:4.20}
Let $\mathcal H$ be a set of primitive characters $\chi$ with
moduli $\leq M$, such that $\mathrm{cond}\, \chi_i \overline{\chi_j}
\leq K$ for any pair $\chi_i, \chi_j$ belonging to $\mathcal H$. Let
$\mathcal S$ be a set of distinct pairs $(\chi_j, \varrho_j)$ with
$L(\varrho_j, \chi_j) = 0$ where $\chi_j \in \mathcal H$, $\beta_j
\geq 1 - h$, $|\gamma_j|\leq T$. ($\chi_i = \chi_j$ is possible, if
$\varrho_i \neq \varrho_j$.)
If $\varepsilon$ is a sufficiently small positive constant, $h <
\varepsilon^3$ then we have for any $K \geq 1$, $M \geq 1$, $T
\geq 2$
\begin{equation}
|\mathcal S| \ll_\varepsilon \big(K^2(MT)^{3/4}\big)^{(1 + \varepsilon)h},
\label{eq:4.18}
\end{equation}
and
\begin{equation}
|\mathcal S| \ll_\varepsilon (K^2 M^2 T^\varepsilon)^{(1 + \varepsilon)h}.\hspace*{6mm}
%%\tag{4.18a}
\label{eq:4.18a}
\end{equation}
\end{Lem}

Finally the following version of the Deuring--Heilbronn
phenomenon, proved in \cite[Theorem 4]{Pin} will be needed in
case of existence of a Siegel zero (see Section~\ref{s:11}).

\begin{Lem}
\label{l:4.21}
Let $\chi_1$ and $\chi_2$ be primitive characters $\mod q_1$ and
$q_2$, resp., with $L(1 - \delta_1, \chi_1) = L(1
- \delta + i \gamma, \chi_2) = 0$, where $\chi_1, \delta_1$ are real,
$\delta_1 < \delta < 1/7$. Let $k = \mathrm{cond}\, \chi_1
\overline \chi_2$, $\varepsilon > 0$, arbitrary,
\begin{equation}
Y = \big(q^2_1 q_2 k(|\gamma|+2)^2\big)^{3/8} \geq Y_0(\varepsilon)
\label{eq:4.19}
\end{equation}
sufficiently large. Then we have
\begin{equation}
\delta_1 \geq (1 - \varepsilon)(1 - 6\delta)\log 2 \cdot
Y^{-(1+\varepsilon)\delta/(1-6\delta)} / \log Y.
\label{eq:4.20}
\end{equation}
\end{Lem}

\section{Minor arcs}
\label{s:5}

The treatment of the minor arcs is completely standard. We will
use the estimate of Vaughan (Lemma~\ref{l:4.10}) on the minor
arcs. This determines the value $3/5$ in our Theorems~\ref{t:2}
and \ref{t:3}.

Using Parseval's identity we obtain from (\ref{eq:1.10}) and
Lemma~\ref{l:4.10}:
\begin{align}
\sum_m R^2_2(m)
&= \int\limits_{\mathfrak m} |\mathcal S^4(\alpha)|d\alpha \label{eq:5.1}\\
&\leq (\max_{\mathfrak m} |\mathcal S(\alpha)|)^2
\int\limits^1_0 |\mathcal S(\alpha)|^2 d\alpha \ll
\max \left({X^2\over P}, X^{8\over 5}\right) X\mathcal L^9.
\nonumber
\end{align}

This result shows that for $m \leq X$ we have
\begin{equation}
|R_2(m)| \leq {X \over \sqrt{\mathcal L}} \ \text{ with }
\ll \mathcal L^{10} \max \left({X \over P}, X^{3/5}\right)
\text{ exceptions,}
\label{eq:5.2}
\end{equation}
\begin{equation}
|R_2(m)| \leq X^{1-\varepsilon} \ \text{ with }
\ll_\varepsilon \max \left({X^{1+3\varepsilon} \over P}, X^{3/5
+ 3\varepsilon}\right)
\text{ exceptions.}
\label{eq:5.3}
\end{equation}

The first inequality will be used if we have no Siegel zero, the
second if we have one. As we can see, the exact choice of $P$
will be irrelevant in (\ref{eq:5.2})--(\ref{eq:5.3}) if we can
choose $P \geq X^{2/5}$ (which will be the case in many
applications).

\section{Basic results about major arcs. Dissection of~{\boldmath$S(\alpha)$}}
\label{s:6}

We will follow \cite{MV} but extend their arguments beyond the Siegel zero to zeros near to
$\sigma = 1$ as well. For $\alpha \in \mathfrak M(q,a)$ let
$\alpha = a / q + \eta$. By $P < X_1$ we have
\begin{equation}
S(\alpha) = {1 \over \varphi(q)} \sum_{\chi(q)}
\chi(a)\tau(\overline \chi)
 S(\chi , \eta) = {1
\over \varphi(q)} \sum_{\chi(q)} \chi (a)\tau(\overline \chi) S(\chi^*, \eta)
\label{eq:6.1}
\end{equation}
where $\chi \mod q$, $q \leq P$ is induced by the primitive character~$\chi^*$, and $S(\chi, \eta)$
is defined by
\begin{equation}
S(\chi, \eta) = \frac1{\varphi (q)} \sum_{X_1 < p \leq X} \chi(p) \log p e(\eta p).
\label{eq:6.1a}
%%\tag{6.1a}
\end{equation}

Using the (unusual) notation of Section~1, we can separate from
$S(\chi^*,\eta)$ the effect of the main term $T_0(\eta)$ `caused' by
the pole of $L(s, \chi_0) = \zeta(s)$ at $s = 1$ and that
of the zeros $\varrho$ lying near to $\sigma = 1$ (for all
$L(s, \chi)$).
Up to the different sign $A(\varrho)$ (see
(\ref{eq:1.26})--(\ref{eq:1.27})) their treatment will be the
same. Accordingly we write
\begin{equation}
S_1(\alpha) = S(\alpha) - S_0(\alpha), \quad
S_0(\alpha) = S_2(\alpha) + S_3(\alpha),
\label{eq:6.2}
\end{equation}
where we define $S_2(\alpha)$ and $S_3(\alpha)$ (and thus
$S_1(\alpha)$ and $S_0(\alpha)$) through (\ref{eq:6.1}) and
$S_i(\chi, \eta)$ $(0 \leq i \leq 3)$ by
\begin{equation}
\aligned
S_2(\chi^*,\eta) &= \sum_{\substack{\varrho = \varrho_\chi \\ H/\mathcal L
< \delta \leq b,\ |\gamma| \leq \sqrt{X}}} A(\varrho)
T(\varrho,\eta), \\
S_3(\chi^*,\eta) &= \sum_{\substack{\varrho = \varrho_\chi \\ 0 \leq \delta
\leq H/\mathcal L,\ |\gamma| \leq \sqrt{X}}} A(\varrho) T(\varrho,
\eta)
\endaligned
\label{eq:6.3}
\end{equation}
where in case of the principal character the pole $\varrho = 1$ with
$A(\varrho) = 1$ is included, $b = b(\eta_0)$ is a small
constant, and for a zero $\varrho$ we have $A(\varrho) = -1$. We remark
that $S_i(\chi, \eta) = S_i(\chi^*, \eta)$. Then we have
\begin{align}
&\sum_{q \leq P} \underset{a(q)}{\sum\nolimits^\prime}\,
\int\limits_{\mathfrak M_{(q, a)}} S^2(\alpha) e(-m\alpha) d\alpha \label{eq:6.4} \\
&= \! \sum_{q \leq P} \underset{a(q)}{\sum\nolimits^\prime}
\sum_{\chi(q)} \sum_{\chi'(q)} \! {\chi\chi'(a)\tau(\overline \chi) \tau(\overline
\chi') e(-a m/q) \over \varphi^2(q)}\! \int\limits^{1/q Q}_{-1/q Q}\!\!
S(\chi, \eta) S(\chi',\eta) e(-m\eta) d\eta \nonumber \\
&=\sum_{q\leq P} \sum_{\chi(q)} \sum_{\chi'(q)}
{c_{\chi\chi'}(-m)\tau(\overline \chi) \tau (\overline \chi')\over
\varphi^2(q)} \int\limits^{1/q Q}_{-1/q Q} S(\chi,\eta) S(\chi', \eta)
e(-m\eta) d\eta \nonumber\\
&\overset{\rm def}{=} \underset{\substack{\chi\\ r(\chi) \leq P}}{\sum\nolimits^*}
\underset{\substack{\chi' \\ r(\chi') \leq P}}{\sum\nolimits^*}
\sum_{\substack{q \leq P
\\ [r(\chi), r(\chi')] \mid q}} c(\chi, \chi', q, m) \int\limits^{1/q
Q}_{-1/q Q} S(\chi,\eta) S(\chi', \eta) e(-m\eta) d\eta.
\nonumber
\end{align}

Naturally the same formula holds if we replace $S(\alpha)$ and
$S(\chi, \eta)$ by $S_i(\alpha)$ and $S_i(\chi, \eta)$, respectively
$(0 \leq i \leq 3)$.

The estimate of these integrals will be performed by the aid of Gal\-lagh\-er's Lemma \ref{l:4.11} through the estimates of the quantities ($\chi$ primitive $\mod r$)
\beq
\label{eq:6.5}
W_i(\chi) := \biggl(\int\limits_{-1/rQ}^{1/rQ} \bigl|S_i(\chi, \eta)\bigr|^2 d\eta \biggr)^{1/2}.
\eeq

\section{Main Lemma. Supplementary singular series}
\label{s:7}

Using the notation of Sections~\ref{s:2} and \ref{s:3} we can
formulate and prove our

\begin{MLem}
Suppose we have two primitive characters $\chi_1 \mod r_1$, $\chi_2
\mod r_2$, $q_0 = [r_1, r_2]$, $q_1 = q_0/(q_0, |m|)$, $\ell_i
= r_i / (r_1, r_2)$, $e = (m, (r_1, r_2))$. Let $\chi^* = (\chi_1
\chi_2)^*$, $\mathrm{cond}\, \chi_1 \chi_2 = r^* = r' \ell_1 \ell_2$,
$b(q) = c(\chi_1, \chi_2, q, m)$,
\begin{equation}
f = \prod_{\substack{p^\alpha \| (r_1, r_2) \\ p\mid m}} p^\alpha,\quad  d =
\prod_{\substack{p^\alpha \|(r_1, r_2) \\ p \nmid m}} p^\alpha,
\label{eq:7.1}
\end{equation}
\begin{equation}
{\mathfrak S}(\chi_1, \chi_2, m) = \sum^\infty_{t = 1} b(q_0 t), \quad
A(\chi_1, \chi_2, m) = \sum^\infty_{t = 1} |b(q_0 t)|.
\label{eq:7.2}
\end{equation}

Suppose $A(\chi_1, \chi_2, m) \neq 0$. Then
\begin{equation}
b(q_0) \neq 0, \quad r' \mid {df \over e} = \frac{(r_1, r_2)}{e},
\label{eq:7.3}
\end{equation}
\begin{equation}
b(q_0) = {\tau(\overline \chi_1) \tau (\overline \chi_2) \tau (\chi^*)
\overline\chi^* \left({-m \over (q_0, |m|)}\right) \mu \left({q_1 \over
r^*}\right)  \chi^* \left( q_1 \over r^*\right) \mu(\ell_1)
\mu(\ell_2) \overline \chi_1(\ell_2) \overline \chi_2(\ell_1) \over
\varphi^2(\ell_1) \varphi^2(\ell_2) \varphi(d) \varphi(f) \varphi(df/e)},
\label{eq:7.4}
\end{equation}
\begin{equation}
|b(q_0)| = {\ell_1 \over \varphi^2(\ell_1)} \cdot {\ell_2\over
\varphi^2(\ell_2)} \cdot{d \over \varphi(d)} \cdot {f\over
\varphi(f)} {\sqrt{r'} \over \varphi(df/e)},
\label{eq:7.5}
\end{equation}
\begin{equation}
{\mathfrak S}(\chi_1, \chi_2, m) = b(q_0) \prod_{\substack{p\nmid m\\ p\nmid [r_1,
r_2]}} \left(1 - {1 \over (p-1)^2} \right) \prod_{\substack{p \mid m \\
p \nmid [r_1, r_2]}} \left(1 + {1 \over (p - 1)} \right),
\label{eq:7.6}
\end{equation}
\begin{equation}
|{\mathfrak S}(\chi_1, \chi_2, m)| \leq {\mathfrak S}(m), \quad
|A(\chi_1, \chi_2, m)| \leq B \cdot |{\mathfrak S}(\chi_1, \chi_2, m)|
\label{eq:7.7}
\end{equation}
with the constant $B = \prod\limits_{p > 2} (1 + 2/(p(p - 2)))$.
Further $|{\mathfrak S}(\chi_1, \chi_2, m)| \leq$ \hbox{$(\sqrt{3}/2) {\mathfrak S}(m)$}
unless the following five relations all hold:
\begin{equation}
{r_i \over (r_i,m)} \mid 36 \  (i = 1,2), \quad
{r_i\over (r_1, r_2)} \mid 3 \  (i = 1,2), \quad r^*\mid 36.
\label{eq:7.8}
\end{equation}
\end{MLem}

In case of the Generalized Twin Prime Problem (see
(\ref{eq:2.4})--(\ref{eq:2.5})) nearly everything remains unchanged.

\begin{MLemm}
If we replace ${\mathfrak S}(\chi_1, \chi_2, m)$ by ${\mathfrak S}'(\chi_1, \chi_2, m)$,
$A(\chi_1, \chi_2, m)$ by the analogous $A'(\chi_1, \chi_2, m)$, and
$\chi^* = (\chi_1 \chi_2)^*$ by $(\chi_1 \overline \chi_2)^*$ then the results
of the Main Lemma~1 hold with the only change that
$\tau(\overline \chi_2)$ and $\overline \chi_2(\ell_1)$ in
{\rm (\ref{eq:7.4})} are to be replaced by $\overline{\tau(\overline
\chi_2)}$ and $\chi_2(\ell)$, respectively.
\end{MLemm}

\begin{MCor}
For the singular series ${\mathfrak S}(\chi_1, \chi_2,m)$ the inequality {\rm
(\ref{eq:1.21})} holds.
\end{MCor}

\begin{MCoro}
Let us replace ${\mathfrak S}(\chi_1, \chi_2, m)$ by\break
${\mathfrak S}'(\chi_1, \chi_2, m)$
in {\rm (\ref{eq:1.21})} and $\mathrm{cond}\, \chi_1 \chi_2$ by
$\mathrm{cond}\, \chi_1 \overline \chi_2$ in {\rm (\ref{eq:1.22})}.
Then the inequality {\rm (\ref{eq:1.21})} remains valid.
\end{MCoro}

The corollaries easily follow by (\ref{eq:7.5}) from the Main
Lemmas 1 and 1'. Since the proof of Main Lemma 1' goes mutatis
mutandis, we will restrict ourselves to the proof of Main
Lemma~1.

\begin{Rem}
In case of $r_1 = r_2 = 1$, we clearly have the classical singular series:
$$
{\mathfrak S}(\chi_0, \chi_0, m) = {\mathfrak S}' (\chi_0, \chi_0, m) = {\mathfrak S}(m)
$$
from (\ref{eq:7.4}) and
(\ref{eq:7.6}).
\end{Rem}

\begin{proof}
Let us investigate an arbitrary non-zero term belonging to
$q = q_0 t = df \ell_1 \ell_2 t$ (with $\chi_0 =
\chi_{0,q}$)
\begin{equation}
b(q_0 t) = \varphi(q)^{-2} c_{\chi_1 \chi_2 \chi_0} (-m) \tau(\overline
\chi_1 \chi_0) \tau (\overline \chi_2 \chi_0) \neq 0.
\label{eq:7.9}
\end{equation}
Let $o_p(n) = \alpha$ if $p^\alpha \| n$. By
Lemma~\ref{l:4.2}, $\tau(\overline \chi_i \chi_0) \neq 0$ implies the relation $p \nmid (q/r_i)$
for $p \mid r_i$. Thus $o_p(r_i)
= o_p(q)$. So we have $(t, [r_1, r_2]) = 1$. For $p \mid
(r_1, r_2)$ we have by the above $o_p(r_1) = o_p(r_2)
= o_p([r_1, r_2]) = o_p(q)$.

If $p\mid r_i$, $p \nmid r_j$ (equivalently $p \mid \ell_i$)
then $\tau (\overline \chi_j \chi_0) \neq 0$ implies by Lemma~\ref{l:4.2}
that by the $\mu$-factor $1 =
o_p\left({q \over r_j}\right) = o_p(q) =
o_p(r_i)$. Similarly we have $|\mu(t)| = 1$.
 Summarizing the above we have
\begin{equation}
|\mu(\ell_1)| = |\mu(\ell_2)| = |\mu(t)| = 1, \quad
(t,q_0) = 1.
\label{eq:7.10}
\end{equation}
If $p\mid \ell_i$ then $o_p(q) = 1$ and $p\mid r^*$. This implies, in view of
(\ref{eq:7.9}), that by Lemma 4.4 we have $p|r^*| q/ (q, |m|)$
and so $p \nmid m$, that is $(m, \ell_i) = 1$ $(i = 1,2)$.
Hence, using the definitions of $d$, $e$, $f$, we have
\begin{equation}
(m,r_1) = (m, r_2) = (m, [r_1, r_2]) = (m, (r_1, r_2)) = (m, df)
= (m, f) = e.
\label{eq:7.11}
\end{equation}
Suppose $A(\chi_1, \chi_2, m) \neq 0$, equivalently there exists a $t$
with (\ref{eq:7.9}). Then, in view of $(t,r^*) = 1$ and
Lemma~\ref{l:4.4}, the equivalent assertions
\begin{equation}
r^* \Big|  {q_0 t\over (q_0 t, |m|)} \Longleftrightarrow r^*
\Big | {q_0 \over (q_0, |m|)} = \ell_1 \ell_2 d {f \over e}
\label{eq:7.12}
\end{equation}
are both true, thus $r' \mid df/e$. Let $j(q) = j_m(q) = {q
\over (q,|m|)}$. Then, by Lemmas \ref{l:4.2} and \ref{l:4.4}, we have
\begin{align}
b(q) ={}& {1 \over \varphi(q)} \cdot {1 \over \varphi(j(q))} \overline \chi^*
\left({-m \over (q, |m|)}\right) \mu\left({j(q) \over r^*}
\right) \chi^* \left({j(q) \over r^*} \right) \tau(\chi^*)\cdot \label{eq:7.13}\\
&
\mu(t\ell_2) \overline \chi_1(t\ell_2) \tau(\overline \chi_1)
\mu(t\ell_1) \overline \chi_2 (t\ell_1)\tau(\overline \chi_2),
\nonumber
\end{align}
where $q = q_0 t = q_0 hk$, $h = \prod\limits_{p\mid t, p\mid m}
p$, $k = \prod\limits_{p\mid t,p \nmid m} t$.
Taking $q = q_0$, that is, $t = 1$, we obtain (\ref{eq:7.4}).
Since $(q_0 hk, |m|) = h(q_0, |m|)$ we have $j(q_0 hk) = kj(q_0)
= kq_1$. Taking into account (\ref{eq:7.10}), we have in case of $b(q_0 t) \neq 0$ from (\ref{eq:7.13})
\begin{equation}
b(q_0 t) = b(q_0) {\chi^*(h)\mu(k) \chi^*(k)\overline \chi_1(kh)
\overline \chi_2(kh) \over \varphi^2(k) \varphi(h)} = b(q_0)
{\mu(k) \over \varphi^2(k)} \cdot {1 \over \varphi(h)}.
\label{eq:7.14}
\end{equation}
Now (\ref{eq:7.14}) shows (\ref{eq:7.6}). Further,
\begin{equation}
\aligned
\sum^\infty_{t=1} |b(q_0t)| &= |b(q_0)| \prod_{p\nmid q_0, p\nmid m}
\left(1 + {1 \over (p-1)^2}\right) \prod_{p\nmid q_0, p\mid m}
\left( 1 + {1 \over p-1}\right)\\
&\leq |{\mathfrak S}(\chi_1, \chi_2, m)| \cdot \prod_{p > 2} \left(\left(1 +
{1\over (p - 1)^2} \right) \Big/ \left(1 - {1\over (p - 1)^2}
\right) \right) \\
&= B|{\mathfrak S}(\chi_1, \chi_2, m)|.
\endaligned
\label{eq:7.15}
\end{equation}
The first equality in (\ref{eq:7.15}) shows $b(q_0)
\neq 0$, when $A(\chi_1, \chi_2, m) \neq 0$, and so by (\ref{eq:7.4})
we have also (\ref{eq:7.5}). Thus it remains to prove
\hbox{$|{\mathfrak S}(\chi_1, \chi_2, m)| \leq {\mathfrak S}(m)$}, and (\ref{eq:7.8}).

Let us investigate the ratio $\xi$ of the two sides $|{\mathfrak S}(\chi_1, \chi_2, m)|$ and ${\mathfrak S}(m)$ separately for each prime. If $p\nmid
[r_1, r_2]$ we have clearly the same factor on both sides. So we
have to study the following cases:

(i) If $p\mid \ell_i$, then by $(m, \ell_i) = 1$ (see (\ref{eq:7.11})) we
have $p \nmid m$, thus $p > 2$.

Now clearly
\begin{equation}
\xi(p) = {p\over (p - 1)^2} : {p(p - 2) \over (p - 1)^2} = {1
\over p - 2} \leq 1.
\label{eq:7.15a}
\end{equation}
Equality holds if and only if $\ell_i = 3$; otherwise $\xi \leq 1/3$.

(ii) Suppose $p\mid d$, then by definition $p\nmid m$, so $p >
2$. Let $p^\alpha \| d$ $(\alpha \geq 1)$, $p^\beta \| r'$. Then
$r' \mid df/e$ implies $0 \leq \beta \leq \alpha$.
Thus writing further on $\xi$ for $\xi(p)$,
\begin{equation}
\xi = {p\over p - 1} \cdot {p^{\beta/2} \over p^{\alpha - 1}(p
- 1)} : {p(p - 2)\over (p - 1)^2} = {p^{1+\beta/2-\alpha} \over
p - 2} \leq {p^{1 - \alpha/2}\over p - 2}.
\label{eq:7.16}
\end{equation}

Now, if $p\geq 5$ we have $\xi \leq \sqrt{5}/3$ for every
$\alpha \geq 1$. Let $p = 3$. Then for $\alpha \geq 3$ we have
$\xi \leq 1/\sqrt{3}$. For $\alpha = 2$, $\beta \leq 1$ we have
$\xi \leq 1/\sqrt 3$. In case of $\alpha = \beta = 2$ we have
$\xi = 1$.

For $\alpha = 1$ $(p = 3)$ we have $3^1\| r_1$, $3^1\|r_2$, so
the $\mod 3$ component of both $\chi_1$ and $\chi_2$ are $\chi_1 \big|_3
= \chi_2 \big|_3 = \chi'$, the only real non-principal character $\mod
3$. Thus $\chi^*\big|_3 = \chi_1 \chi_2 \big|_3 = \chi_0$, and consequently $3
\nmid r^*$, $\beta = 0$. In this case we have again equality in (\ref{eq:7.16}).
Summarizing, we have equality in (\ref{eq:7.16})
if and only if $d = 3$, $3^1\| r_1$, $3^1\|r_2$ or $d = 9$ and $3^2\|r' \Leftrightarrow
3^2\|r^*$.

Otherwise $\xi \leq \sqrt{5}/3$.

(iii) Finally if $p \,|\, f$, then by definition $p\,|\, e$, $p\,|\,
m$. Let $p^\alpha \| f/e$, $p^\beta \| r'$ $(0 \leq \beta \leq
\alpha)$. Then
\begin{equation}
\xi = {p \over p - 1} \cdot {p^{\beta/2} \over
\varphi(p^\alpha)} : {p \over p - 1} = {p^{\beta/2} \over
\varphi(p^\alpha)} \leq {p^{\alpha/2} \over \varphi(p^\alpha)}.
\label{eq:7.16a}
\end{equation}

If $\alpha = 0$ then clearly $\beta = 0$ and $\xi = 1$ (for
every $p$). Let us suppose $\alpha \geq 1$. If $p\geq 3$ then
$\xi \leq \sqrt 3/2$. Let $p = 2$.
Then for $\alpha \geq 3$ we have $\xi \leq 1/\sqrt 2$. For
$\alpha = 2$, $\beta \leq 1$ we have $\xi \leq 1/\sqrt 2$. In
case of $\alpha = \beta = 2$ we have $\xi = 1$.
If $\alpha = 1$ there is no non-principal character $\mod  2$,
so $\beta = 0$ and $\xi = 1$.
Summarizing, $\xi = 1$ holds if and only if $\alpha = \beta = 0$, $p$
arbitrary, that is $p \nmid f/e$ or
\begin{equation}
p = 2, \ \alpha = 1, \ \beta = 0 \ \text{ or } \ p = 2,\ \alpha
= \beta = 2,
\label{eq:7.17}
\end{equation}
that is
\begin{equation}
2\| f/e, \ 2 \nmid r' \Leftrightarrow 2 \nmid r^* \ \text{ or } \
2^2\| f/e, \ 2^2\|r' \Leftrightarrow 2^2\| r^*.
\label{eq:7.18}
\end{equation}
Otherwise $\xi \leq \sqrt 3/2$.

The considerations (i), (ii), (iii) really show that we have
always
$$
\bigl|{\mathfrak S}(\chi_1, \chi_2, m)\bigr| \leq {\mathfrak S}(m).
$$
Further,
$$
\bigl|{\mathfrak S}(\chi_1, \chi_2, m)\bigr| \leq (\sqrt 3/2) {\mathfrak S}(m)
$$
unless (\ref{eq:7.8}) holds.
\end{proof}

\section{Reduction for zeros near to {\boldmath $\sigma = 1$}}
\label{s:8}

In this section we will show (using the notation of Section~6) that error terms arising from
$S^2_1$ and $S_1 S_0$ make a contribution of
\begin{equation}
O(\mathcal L^{8}X^{1 - b/82})
\label{eq:8.1}
\end{equation}
to $R_1(m)$. Thus, further on, it is enough to study the
integral containing $S^2_0$. First we estimate the term with $S^2_1$.
Using the notation from Sections~1, \ref{s:3} and
\ref{s:6} by Lemmas \ref{l:4.1}--\ref{l:4.2} and (\ref{eq:6.1}) we have,
with the definition of $W_1(\chi)$ in \eqref{eq:6.5}
\begin{align}
&\Biggl|\sum_{q \leq p} \underset a{\sum\nolimits^\prime}
\int\limits_{\mathfrak M(q,a)} S^2_1(\alpha) e(-m \alpha)d\alpha
\Biggr| \leq \label{eq:8.2}\\
& \leq  \sum_{q \leq P} \underset {a(q)}{\sum\nolimits^\prime}
\int\limits_{\mathfrak M(q,a)} |S^2_1(\alpha)|d\alpha =\nonumber\\
& =  \sum_{q \leq p} \underset {a(q)}{\sum\nolimits^\prime}
\int\limits^{1/qQ}_{-1/qQ} \left| {1 \over \varphi(q)}
\sum_{\chi(q)} \chi(a) \tau (\overline \chi) S_1(\chi, \eta)\right|^2 d\eta
= \nonumber\\
& =  \sum_{q \leq P} {1 \over \varphi^2(q)} \sum_{\chi(q)}
\sum_{\chi'(q)} \tau (\overline \chi) \overline{\tau(\overline \chi')}
\underset {a(q)}{\sum\nolimits^\prime} \chi(a) \overline \chi'(a)
\int\limits^{1/qQ}_{-1/qQ} S_1(\chi,\eta) \overline{S_1(\chi', \eta)}
d\eta = \nonumber\\
& =  \sum_{q \leq P} {1 \over \varphi(q)} \sum_{\chi(q)}
|\tau(\overline \chi)|^2 \int\limits^{1/q Q}_{-1/qQ}
|S_1(\chi,\eta)|^2 d\eta =\nonumber\\
& =  \sum_{q\leq P} {1 \over \varphi(q)} \sum_{\chi(q)}
|\tau(\overline \chi)|^2 \int\limits^{1/qQ}_{-1/qQ} |S_1(\chi^*,
\eta)|^2 d\eta \leq \nonumber\\
& \leq  \sum_{r \leq P} \underset{\chi(r)}{\sum\nolimits^*} r
\int\limits^{1/rQ}_{-1/rQ} |S_1(\chi, \eta)|^2 d\eta
\sum_{\substack{\ell \leq P/r\\ (\ell, r) = 1}} {1 \over \varphi(r\ell)} \leq \nonumber\\
& \leq  \sum_{r \leq P} {r \over \varphi(r)}
\underset{\chi(r)}{\sum\nolimits^*} (W_1(\chi))^2 \sum_{\ell \leq P/r}
{1\over \varphi(\ell)} \ll \mathcal L^2 \sum_{r \leq P}
\underset{\chi(r)}{\sum\nolimits^*} W^2_1(\chi).\nonumber
\end{align}

As we can see, at the cost of a logarithm we could get rid of all
cross-products $S_1(\chi, \eta) \overline{S_1(\chi', \eta)}$ with $\chi \neq \chi'$.
The loss of the logarithm would
be crucial near $\sigma = 1$ but not here. We can estimate
$W_1(\chi)$ ($\chi$ primitive $\mod r$, $1 \leq r \leq P$) by means of
Gallagher's lemma (Lemma~\ref{l:4.11}) as follows.
\begin{equation}
W^2_1(\chi) \ll \int\limits^X_{X_1-Y} \bigg| {1 \over Y} \sum_{\substack{x <
n \leq x + Y \\ X_1 \leq n \leq X}} a_n \bigg|^2 dx \leq
I_1(\chi) + I_2(\chi) + I_3(\chi),
\label{eq:8.3}
\end{equation}
where $X_2 = \max(X_1, 6Y)$,
\begin{equation}
Y = rQ/2(\leq X/2), \ \
I_1 = \int\limits^{X_2}_{X_1 - Y}, \ \
I_2 = \int\limits^{X-Y}_{\min(X_2, X-Y)}, \ \
I_3 = \int\limits^X_{X-Y}
\label{eq:8.4}
\end{equation}
(where $I_2$ is missing if $Y \geq X/7$) and with the notation
(\ref{eq:1.26})--(\ref{eq:1.27})
\begin{equation}
a_n = \begin{cases}
\chi(p) \log p - b_n &\text{if } n = p\\
-b_n & \text{if } n \neq p\end{cases},
\quad
b_n = \underset{\substack{\varrho=\varrho_\chi\\
0 \leq \delta \leq b, \ |\gamma | \leq \sqrt{X}}}{\sum\nolimits^\prime}
A (\varrho)n^{\varrho-1}. \label{eq:8.5}
\end{equation}
The dash at the summation sign means that the summation is extended for $\varrho = 1$ in case of $\chi(\mod 1)$.

The treatment of the two tails, $I_1$ and $I_3$ are simpler and
basically the same. Using the explicit form of $\psi(x, \chi)$ (see
(\ref{eq:4.16})) we obtain for any $x\in [X - Y, X]$, in view of Lemma~4.15
\begin{align}
{1 \over Y} \sum_{x\leq n \leq X} a_n &= {-1\over Y}
\sum_{\substack{\varrho = \varrho_\chi \\ b < \delta \leq 1/2,\ |\gamma|
\leq \sqrt{X}}} {X^\varrho - x^\varrho \over \varrho} + O
\left({\sqrt X \over Y} \mathcal L^2\right) \label{eq:8.6}\\
&\ll \sum_{\substack{\varrho = \varrho_\chi \\ b < \delta \leq 1/2, |\gamma| \leq \sqrt X}} \min
\left( {X^{1-\delta} \over Y(|\gamma|+1)}, X^{-\delta}\right) +
O \left( {\sqrt X \over Y} \mathcal L^2 \right).\nonumber
\end{align}

The effect of the last error terms of the form $\mathcal L^2
\sqrt{X}/Y$ is, after squaring, integrating and summing for
all characters,
\begin{equation}
\ll \mathcal L^4 \sum_{r \leq P} r  {X  rQ \over
(rQ)^2} = {\mathcal L^4 XP\over Q} = \mathcal L^4 P^2.
\label{eq:8.7}
\end{equation}

We divide the remaining zeros into $\ll \mathcal L^3$
classes according to their real, imaginary parts and the conductor
$r$ of the relevant primitive character as follows:
\begin{equation}
r\! \sim\! R, \ (2^\mu\! - \! 1) {X \over RQ} \leq |\gamma| \leq
(2^{\mu+1} - 1) {X \over RQ},\  h_\nu\! -\! {1 \over\mathcal L} \leq
\delta \leq h_\nu,\ \ h_\nu = {\nu \over \mathcal L}
\label{eq:8.8}
\end{equation}
where
\begin{equation}
2^k = R \leq P/2,\ \ \mu = 0,1,\dots [\log \sqrt{X} / \log 2], \ \
b\mathcal L \leq \nu \leq \lceil \mathcal L/2\rceil .
\label{eq:8.9}
\end{equation}
Let us denote the contribution of any given class $(R, \mu,
\nu)$ to $\sum\limits_r \underset{\chi(r)}{\sum\nolimits^*}I_3(\chi)$
(with the notation $2^{\mu} = M$) by $J_3(R,M,h)$.
Then by the Cauchy--Schwarz inequality and $X/Y = 2P/r$ we have (the
conditions $R\geq 1$, $M \geq 1$  will be omitted)
\begin{align}
&\sum_{r \leq P} \underset{\chi(r)}{\sum\nolimits^*} I_3(\chi) \ll
\mathcal L^6 \max_{\substack{R\leq P, M \leq X\\ b\leq h\leq 1/2}} J_3(R,M,h)
\label{eq:8.10}\\
&\ll \mathcal L^6 \max_{\substack{R\leq P, M \leq X\\ b\leq h\leq 1/2}}
\sum_{r\sim R} \underset{\chi(r)}{\sum\nolimits^*} N^2\left(1 - h,
{XM\over RQ}, \chi\right) \left({X^{-h}\over M}\right)^2 \cdot RQ \nonumber\\
&\ll \mathcal L^6 \max_{\substack{R\leq P, M \leq X \\ b\leq h \leq
1/2}} {XM \over RQ} \mathcal L{RQ \over M} \cdot M^{-1} N^*
\left(1 - h, {XM\over RQ}, 2R\right) \cdot X^{-2h} \nonumber \\
&= X \mathcal L^7 \max_{\substack{R\leq P, M \leq X \\ b\leq h \leq 1/2}}
M^{-1} N^* \left(1 - h, {XM \over RQ}, 2R\right) \cdot X^{-2h}.\nonumber
\end{align}

If $h \leq 3/8 - \varepsilon$ we apply the imperfect density
theorem of Heath--Brown (Lemma~\ref{l:4.18}) and obtain
\begin{align}
M^{-1} N^*\left(\! 1\! -\! h, {XM\over RQ}, 2R\! \right) X^{-2h}
&\ll
\left(R^2 {P^{6/5}\over R^{6/5}} \right)^{\left({20\over 9} +
\varepsilon \right)h} M^{\left({8\over 3} + {6\over 5}
\varepsilon\right) h - 1} X^{-2h} \label{eq:8.11}\\
\ll \big(X^{-1} P^{{20\over 9} + \varepsilon}\big)^{2h}
&\ll X^{-\left(1 - \left( {20\over 9} +
\varepsilon\right)\vartheta\right)2b} \ll X^{-b/41}.
\nonumber
\end{align}

If $3/8 - \varepsilon \leq h \leq 1/2$ we will use
Lemma~\ref{l:4.15a}. Then we have by $3h \leq 1 + h$
\begin{align}
M^{-1}N^* \left(1 - h, {XM\over RQ}, R\right) X^{-2h}
&\ll \mathcal L^9 \left(R^2 \cdot {P\over R} \right)^{3h\over 1+h}
M^{{3h\over 1+h} - 1} X^{-2h} \label{eq:8.12}\\
\ll \mathcal L^9\big(P^{3\over 1+h} X^{-1}\big)^{2h}
&\ll \mathcal L^9 \cdot X^{\left(\left({24\over 11} +
2\varepsilon\right) \vartheta - 1\right) (3/4 - 2\varepsilon)}
\ll X^{-1/45} .
\nonumber
\end{align}

Since the estimation of $I_1$ runs completely analogously,
\begin{equation}
\sum_{r\leq P} \underset{\chi(r)}{\sum\nolimits^*} (I_1(\chi ) +
I_3(\chi)) \ll \mathcal L^7 X^{1 - b/41}.
\label{eq:8.13}
\end{equation}

Suppose now that $X_2 < X - Y$, that is $Y < X/7$, otherwise we
are ready. If $x \in (X_2, X - Y)$, then $x \geq 6Y$ and
\begin{equation}
[x, x + Y] \subset [X_1, X].
\label{eq:8.14}
\end{equation}
Thus the condition $X_1 < n \leq X$ can be omitted in (\ref{eq:8.3}).
So let us suppose that $Y \leq x/6$ and consider with the
notation (\ref{eq:8.5})
\begin{equation}
I'_2(\chi, x) = Y^{-2} \int\limits^{2x}_x |\vartheta(u + Y) -
\vartheta(u)|^2 du,
\quad \vartheta(u) = \sum_{n \leq u} a_n.
\label{eq:8.15}
\end{equation}

For this integral we can apply the idea of Saffari and Vaughan
\cite{SV}, to replace $u + Y$ by $u + \theta u$. Although the proof
runs completely analogously to \cite[Lemma~6]{SV}, for the sake of
completeness we will present their arguments here, since our
function $\vartheta(u)$ is different now.

Suppose that $2Y \leq v \leq 3Y$, $x \leq u \leq 2x$. In this
case we have $Y \leq v - Y \leq 2Y$, $x \leq u + Y \leq u + v \leq 3x$. Further
\begin{equation}
|\vartheta(u + Y) - \vartheta(u)|^2 \leq 2 \big(|\vartheta(u +
v) - \vartheta(u)|^2 + |\vartheta(u + Y + v - Y) - \vartheta (u
+ Y)|^2\big).
\label{eq:8.16}
\end{equation}
Thus on the right-hand side the starting points of the intervals
are in $[x,3x]$ and the length is in $[Y, 3Y]$. So we can write
(\ref{eq:8.16}) for all possible values of $v\in (2Y, 3Y)$ for
any $u$ to obtain
\begin{align}
Y \int\limits^{2x}_x |\vartheta(u + Y)\! -\! \vartheta(u)|^2 du
&\leq
4 \int\limits^{3x}_x \int\limits^{3Y}_Y |\vartheta (u + v') -
\vartheta(u) |^2 dv'\,du =\label{eq:8.17}\\
&= 4 \int\limits^{3x}_x \int\limits^{3Y/u}_{Y/u} |\vartheta(u +
\theta u) - \vartheta(u)|^2 ud \theta du \leq \nonumber\\
&\leq 4\cdot 3x \int\limits^{3x}_x \int\limits^{3Y/x}_{Y/2x}
|\vartheta (u + \theta u) - \vartheta (u)|^2 d\theta du = \nonumber \\
&= 12 x \cdot \int\limits^{3Y/x}_{Y/2x} \bigg(\int\limits^{3x}_x
|\vartheta(u + \theta u) - \vartheta(u)|^2 du\bigg) d\theta \leq \nonumber\\
&\leq 30 Y \max_{Y/2x \leq \theta \leq 3Y/x} \int\limits^{3x}_x
|\vartheta(u + \theta u) - \vartheta(u)|^2 du. \nonumber
\end{align}
Hence,
\begin{equation}
I'_2(\chi, x) \leq 30 Y^{-2} \max_{Y/2x \leq \theta \leq 3Y/x}
\int\limits^{3x}_x |\vartheta(u+\theta u) - \vartheta(u)|^2du.
\label{eq:8.18}
\end{equation}
Similarly to (\ref{eq:8.6}) we have
\begin{equation}
\vartheta (u + \theta u) - \vartheta(u) = {-1 \over Y}
\underset{\substack{\varrho = \varrho_\chi \\ 1/2 \geq \delta \geq b, \
|\gamma| \leq \sqrt{X}}}{\sum\nolimits^\prime} {u^\varrho((1 +
\theta)^\varrho - 1) \over \varrho} + O \left({\sqrt x \mathcal
L^2 \over Y}\right).
\label{eq:8.19}
\end{equation}
The contribution coming from the term $\mathcal L^2 \sqrt x/Y$
towards the final value of $\sum\limits_{r \leq P}
\underset{\chi(r)}{\sum\nolimits^*} I_2(x)$ will be similar to (\ref{eq:8.7}):
\begin{equation}
\ll \mathcal L^4 \sum_{\substack{x = 2^\nu \\ 2^\nu \leq X}} \sum_{r\leq
P} r \cdot {x \cdot rQ \over r^2 Q^2} \ll {\mathcal L^5 X P \over Q} =
\mathcal L^5 P^2.
\label{eq:8.20}
\end{equation}
Using the trivial inequality $0 \leq \theta \leq 1$
\begin{equation}
{(1 + \theta)^\varrho - 1 \over \varrho} \ll \min \left(\theta,
{1 \over |\varrho|}\right)
\label{eq:8.21}
\end{equation}
we obtain after squaring and integration in (\ref{eq:8.19}),
abbreviating the summation conditions by $\sum''$, for the term
${I_2}''(\chi, x)$ containing the zeros, the following inequality:
\begin{align}
I''_2(\chi, x) &\ll Y^{-2} \underset{\varrho}{\sum\nolimits^{\prime\prime}}
\underset{\varrho'}{\sum\nolimits^{\prime\prime}}
{|x^{\varrho + \overline \varrho' + 1}| \over |\varrho +
\overline \varrho'+1|} \min \left(\theta, {1\over
|\varrho|}\right) \min \left( \theta, {1\over |\varrho'|}\right)
\label{eq:8.22}\\
&\ll Y^{-2}\underset{\delta' \geq \delta}
{\underset{\varrho}{\sum\nolimits^{\prime\prime}}
\underset{\varrho'}{\sum\nolimits^{\prime\prime}}} {\theta x^{3
- \delta - \delta'}\over 1 + |\gamma - \gamma'|} \min
\left(\theta, {1 \over |\varrho|} \right) \nonumber \\
&\ll Y^{-1} \underset{\varrho}{\sum\nolimits^{\prime\prime}}
\mathcal L^2 x^{2 - 2\delta} \min \left(\theta, {1 \over |\varrho|}\right).
\nonumber
\end{align}

Using the same classification of moduli and zeros as in
(\ref{eq:8.9}) (with $x$ in place of $X$) we obtain by (\ref{eq:8.11}) and \eqref{eq:8.12}
\begin{align}
\sum_{r\leq x/Q}\underset{\chi(r)}{\sum\nolimits^*} I''_2(\chi,x)
&\ll \mathcal L^5 Y^{-1}\cdot {Y\over x} \max_{\substack{R \leq
P, \ 1 \leq M \leq X\\ b \leq h \leq 1/2}} M^{-1} N^*
\! \left( \!1 \! - \!
h, {xM \over RQ}, 2R \! \right)\! x^{2 - 2h}\label{eq:8.23}\\
&\leq \mathcal L^5 \max_{\substack{R \leq
P, \ 1 \leq M \leq X\\ b \leq h \leq 1/2}} M^{-1} N^* \left(1 -
h, {XM \over RQ}, 2R \right) X^{1 - 2h}\nonumber\\
&\ll \mathcal L^5 X^{1-b/41}.
\nonumber
\end{align}
Summing over $x = 2^\nu$, $X_2/2 \leq 2^\nu \leq X$ we
finally have from (\ref{eq:8.20})--(\ref{eq:8.23})
\begin{equation}
\sum_{r\leq P} \underset{\chi (r)}{\sum\nolimits^*} I_2(\chi ) \ll
\mathcal L^6 X^{1-b/41} + \mathcal L^5 P^2 \ll \mathcal L^6 X^{1-b/41}.
\label{eq:8.24}
\end{equation}
This together with (\ref{eq:8.2})--(\ref{eq:8.4}) and
(\ref{eq:8.13}) gives the estimate
\begin{equation}
\int\limits_{\mathfrak M} |S^2_1(\alpha)|d\alpha \ll \mathcal L
\sum_{r \leq P} \underset{\chi (r)}{\sum\nolimits^*} W^2_1(\chi ) \ll
\mathcal L^{8} X^{1-b/41}.
\label{eq:8.25}
\end{equation}

Since the above arguments were valid for any $b \geq 0$, we have mutatis mutandis
\begin{equation}
\int\limits_{\mathfrak M}|S_0(\alpha)|^2 d\alpha \ll \mathcal
L^{8} X.
\label{eq:8.26}
\end{equation}
Thus, together with (\ref{eq:8.25}), we obtain by the Cauchy--Schwarz inequality
\begin{equation}
\int\limits_{\mathfrak M}|S_0(\alpha) S_1(\alpha)| d\alpha \ll \mathcal
L^{8} X^{1-b/82}.
\label{eq:8.27}
\end{equation}

Summarizing, we proved
\begin{equation}
R_1(m) = \int\limits_{\mathfrak M} S^2(\alpha) e(-m\alpha)
d\alpha = \int\limits_{\mathfrak M} S^2_0(\alpha) e(-m\alpha)
d\alpha + O(\mathcal L^{8} X^{1-b/82}).
\label{eq:8.28}
\end{equation}

\section{Reduction to generalized exceptional zeros}
\label{s:9}

We will continue with the investigation of $S^2_0 = (S_2 + S_3)^2$
and show that the contribution of $S^2_2$ and $S_2 S_3$ to $R_1(m)$ are both
\begin{equation}
O_{\eta_0}({\mathfrak S}(m) e^{-c(\eta_0)H} X).
\label{eq:9.1}
\end{equation}
If there is no Siegel zero, then (\ref{eq:8.1}) and (\ref{eq:9.1})
will imply that the study of $S(\alpha)$ on the major arcs can be
restricted to that of $S_3(\alpha)$. $S_3(\alpha)$ contains only a
bounded number of terms, since by Lemma~\ref{l:4.17}, there are
only $c(\eta_0)e^{CH}$ zeros in the definition of $S_3(\alpha)$. If
there is a Siegel zero then we need an estimate sharper than
(\ref{eq:9.1}). This will be made possible by the
Deuring--Heilbronn phenomenon (Lemma~\ref{l:4.21}). This shows
that a part of the region, associated with the definition of
$S_2(\alpha)$ will be free of zeros of any $L$-functions
with a primitive character modulo any $r \leq P$.

Now we have to be more careful than in Section~8, because it is
not allowed to loose any logarithms. First we consider $S^2_2$.
 By the Main Lemma~1 we have with the notation of Section~2 and $r(\chi) = \mathrm{cond}\, \chi$, $r(\chi') = r'$,
 $B = \prod\limits_{p > 2} (1 + 2 / p(p - 2))$, similarly to \eqref{eq:6.4}, with $W_2(\chi)$ defined by \eqref{eq:6.5}
\begin{align}
&\bigg| \sum_{q \leq p} \underset a{\sum\nolimits^\prime}
\int\limits_{\mathfrak M(q,a)} S^2_2(\alpha) e(-m\alpha)d\alpha
\bigg|
\label{eq:9.2}\\
&=  \bigg| \underset{r(\chi) \leq P}{\sum\nolimits^*}\ \underset{r(\chi')\leq
P}{\sum \nolimits^*} \sum_{\substack{q \leq P\\ [r(\chi),
r(\chi')]\mid q}} c(\chi, \chi', q, m) \int\limits^{1/qQ}_{-1/qQ}
S_2(\chi,\eta) S_2(\chi', \eta) e(-m\eta) d\eta\bigg|\nonumber\\
&\leq \underset{r(\chi) \leq P}{\sum\nolimits^*}\ \underset{r(\chi')
\leq P}{\sum\nolimits^*} \sum^\infty_{\substack{q = 1\\ [r(\chi),
r(\chi')]\mid q}} |c(\chi,\chi', q,m)|
\int\limits^{1/Q[r,r']}_{-1/Q[r,r']} |S_2(\chi,\eta)|\,
|S_2(\chi',\eta)| d\eta\nonumber\\
&\leq  \underset{r(\chi) \leq P}{\sum\nolimits^*}\,
\underset{r(\chi')\leq P}{\sum\nolimits^*}\!\! B{\mathfrak S}(\chi,\chi', m) \bigg(
\int\limits^{1/rQ}_{-1/rQ} \!\! |S_2(\chi,\eta)|^2 d\eta \! \bigg)^{\!\! 1/2}
\! \bigg( \int\limits^{1/r'Q}_{-1/r'Q} \!\! |S_2(\chi',
\eta)^2|d\eta \! \bigg)^{\!\! 1/2} \nonumber\\
&\leq  B {\mathfrak S}(m) \bigg(\sum_{r\leq P}
\underset{\chi(r)}{\sum\nolimits^*} W_2(\chi)\bigg)^2.
\nonumber
\end{align}

We will treat $W_2(\chi)$ similarly, but somewhat simpler than $W_1(\chi)$ in Section~8.
For example, the tails will be estimated the same way
as the essential part. The Dirichlet series appearing in the
definition of $S_2(\chi,\eta)$ is now (cf.\
(\ref{eq:6.1})--(\ref{eq:6.3}))
\begin{equation}
b'_n = \underset{\varrho}{\sum\nolimits^+} A(\varrho) n^{\varrho - 1},
\label{eq:9.3}
\end{equation}
where by $\sum\nolimits^+$ we denote the summation conditions
\[
\varrho = \varrho_\chi, \ H/\mathcal L <  \delta \leq b,\
|\gamma|\leq \sqrt{X},
\]
where $H$ will be a large constant to be chosen later.

In order to estimate $\sum\sum^* W_2(\chi )$ by Gallagher's lemma (Lemma \ref{l:4.11})
let us consider first
for a fixed $\chi$ an arbitrary interval of type ($x$, $x + y$), where
\begin{equation}
1 \leq x \leq X, \ \ 1 \leq y \leq Y = rQ/2,
\label{eq:9.4}
\end{equation}
and apply again Gallagher's lemma (Lemma~\ref{l:4.11}).
Then we have for any $\chi $ by Lemma~\ref{l:4.7}
\begin{align}
{1\over Y} \sum^{x+y}_{n=x} b'_n &= \underset {\varrho\phantom{+}}
{\sum\nolimits^+} \left\{ {(x + y)^\varrho - x^\varrho \over
\varrho Y} + O \left({1\over Y}\right) \right\}\label{eq:9.5}\\
&\ll \underset {\varrho\phantom{+}}{\sum\nolimits^+} x^{-\delta} \min \left(
{y\over Y}, {X \over |\varrho|Y}\right) + Y^{-1} N(1 - b, \sqrt X, \chi).
\nonumber
\end{align}
The total contribution of the last error term to $\sum\limits_{r\leq P} \underset{\chi(r)}{\sum\nolimits^*} W_2(\chi)$
after squaring, summing and integrating will be by Lemma~\ref{l:4.15a} for any $b\leq 1/4$
\begin{equation}
\ll \sqrt{X} \mathcal L^C \max_{R \leq P} (RQ)^{-1} N^*
\left({3\over 4}, \sqrt{X}, R\right) \ll {\sqrt{X} \mathcal L^C
P^{1/5} X^{3/10} \over Q} \ll \mathcal L^C X^{1/3},
\label{eq:9.6}
\end{equation}
which is negligible.

Denoting the contribution of zeros after squaring, integrating and summing
to $W_2(\chi)$ by $W'_2(\chi)$, let
us define the positive coefficients
\begin{equation}
a_\varrho = \min \left(1, {X\over QR|\varrho|}\right) = \min
\left( 1, {P\over R|\varrho|}\right) \quad \text{ for } r \in [R, RX^\varepsilon].
\label{eq:9.7}
\end{equation}
Then if $b\leq 1/4$ we have $\delta \leq 1/4$ and so
\begin{equation}
(W'_2(\chi ))^2 \ll \int\limits^X_1 \bigg(\sum_\varrho a_\varrho
x^{-\delta}\bigg)^2 dx = \sum\sum a_\varrho a_{\varrho'}
\int\limits^X_1 x^{-\delta-\delta'} dx.
\label{eq:9.8}
\end{equation}
Hence
\begin{equation}
W'_2(\chi ) \ll \bigg(\sum_\varrho \sum_{\varrho'} a_\varrho a_{\varrho'}
X^{1-\delta-\delta'}\bigg)^{1/2} = X^{1/2} \sum a_\varrho
X^{-\delta}.
\label{eq:9.9}
\end{equation}
Let us consider now the contribution of all zeros $\varrho =
\varrho_\chi$, $\mathrm{cond}\, \chi = r$ with the property
\begin{equation}
(2^\mu - 1) {P\over R_\nu} \leq |\gamma| \leq (2^{\mu+1} - 1)
{P\over R_\nu}, \quad r \in [R_\nu, R_\nu X^\varepsilon], \ R_\nu = X^{\nu\varepsilon} \leq P,
\label{eq:9.10}
\end{equation}
to $\sum\sum W'_2(\chi)$, where $\varepsilon$ is a small absolute
constant, to be chosen later, depending on $\eta$.
Let $M_\mu = (2^{\mu+1} - 1) = \left[ {2\sqrt X R_\nu\over P}\right]$.
Let us fix now the constant $b = b(\eta_0) \leq 1/6$ in such
a way that with the notation
\begin{align}
c_2(\delta) &= {3 \over 2(1 - 4\delta)} < {3\over 4}
\left({2\over 1 - 4\delta} + {1 \over (1 - 2\delta)(1 - 4\delta)}  \right)=\label{eq:9.10a}\\
& = {3(3 - 4\delta)\over 4(1 - 4\delta)(1 - 2\delta)} = c_1(\delta),
\nonumber
\end{align}
the relation
\[
c_3(\delta) = 1 - \left({4\over 9} - \eta\right) c_1(\delta) > 0
\]
should hold for $0 \leq \delta \leq b$ (that is, for $\delta = b$),
and apply Lemma~\ref{l:4.19}.
From (\ref{eq:9.7})--(\ref{eq:9.10a})  in view of
$\delta c_2(\delta) \leq bc_2(b) \leq 3/4$, we obtain by partial integration with respect to $\delta$ the inequality
\begin{eqnarray}
&&\hspace*{-7.5mm} X^{-1/2} \sum_{r\leq P}
\underset{\chi(r)}{\sum\nolimits^\prime } W'_2(\chi) \label{eq:9.11}\\
& \ll & \sum_{R_\nu \leq P} \sum_{M_\mu \leq X} M^{-1}_\mu
\int\limits^b_{H/\mathcal L} X^{-\delta} d_\delta N^*\left(1 -
\delta, {PM_\mu \over R_\nu}, R_\nu X^\varepsilon\right) d\delta
\nonumber \\
&\ll_\varepsilon & \max_{R_\nu \leq P} \sum_{M_\mu \leq X}
M^{-1}_\mu \Bigg\{ X^{-b} N^* \left(1 - b, {PM_\mu \over R_\nu},
R_\nu X^\varepsilon\right) \nonumber\\
& & + \mathcal L
\int\limits^b_{H/\mathcal L} N^* \left(1 - \delta, {PM_\mu\over
R_\nu}, R_\nu X^\varepsilon\right) X^{-\delta} d\delta \Bigg\}\nonumber \\
&\ll_\varepsilon & \max_{R_\nu < P} \sum_{M_\mu \leq X}
M^{-1+bc_2(b)(1+\varepsilon)}_\mu
\bigg\{ \left( R_\nu^{c_1(b)} {P^{c_2(b)}\over R^{c_2(b)}_\nu}
X^{-1+3\varepsilon} \right)^b \nonumber\\
& & +\mathcal L
\int\limits^b_{H/\mathcal L} \left(R_\nu^{c_1(\delta)}
{P^{c_2(\delta)} \over R_\nu^{c_2(\delta)}} \cdot
X^{-1+3\varepsilon}\right)^\delta d\delta\bigg\}\nonumber \\
&\ll_\varepsilon & \big(P^{c_1(b)} X^{-1+3\varepsilon}\big)^b +
\mathcal L \int\limits^b_{H/\mathcal L} (P^{c_1(\delta)}
X^{-1+3\varepsilon})^\delta d\delta\nonumber \\
&\ll_\varepsilon & X^{-(c_3(b) - 3\varepsilon)b} + \mathcal L
\int\limits^b_{H/\mathcal L} X^{-(c_3(b)-3\varepsilon)\delta}
d\delta \nonumber \\
&\ll_\varepsilon & {1\over c_3(b) - 3\varepsilon}
e^{-(c_3(b) - 3\varepsilon)H} \ll_{\eta_0} e^{-c_4(\eta_0)H} .
\nonumber
\end{eqnarray}
Hence, from (\ref{eq:9.2}) we get
\begin{equation}
\int\limits_{\mathfrak M} S^2_2(\alpha) e(-m\alpha) d\alpha
\ll_{\eta_0} {\mathfrak S}(m) e^{-2c_4(\eta_0)H}.
\label{eq:9.12}
\end{equation}
We can repeat the same procedure as above for
$S_3(\alpha)$ in place of $S_2(\alpha)$ to obtain the same
result with $H = 0$, that is
\begin{equation}
X^{-1/2} \sum_{r\leq P} \underset{\chi(r)}{\sum\nolimits^*} W_3(\chi) \ll_{\eta_0} {\mathfrak S}(m).
\label{eq:9.13}
\end{equation}

Analogously to (9.2) we can estimate
\begin{equation}
\aligned
\int\limits_{\mathfrak M} S_2(\alpha) S_3(\alpha) e(-m\alpha)d\alpha
&\ll {\mathfrak S}(m) \bigg(\underset{r(\chi)\leq P}{\sum\nolimits^*}
W_2(\chi)\bigg) \bigg(\underset{r(\chi )\leq P}{\sum\nolimits^*}
W_3(\chi ) \bigg)\\
&\ll_{\eta_0} {\mathfrak S}(m) e^{-c_4(\eta_0)H} X.
\endaligned
\label{eq:9.14}
\end{equation}

Summarizing, we have from (\ref{eq:9.12}) and (\ref{eq:9.14})
\begin{equation}
\int\limits_{\mathfrak M} S^2_0(\alpha) e(-m\alpha) d\alpha =
\int\limits_{\mathfrak M} S^2_3(\alpha)e(-m\alpha)d\alpha +
O_{\eta_0} ({\mathfrak S}(m)e^{-c(\eta_0)H} X).
\label{eq:9.15}
\end{equation}

\section{Effect of the generalized exceptional zeros}
\label{s:10}

Finally we examine the crucial part of the contribution of
the major arcs, namely
\begin{equation}
\int\limits_{\mathfrak M} S^2_3(\alpha) e(-m\alpha)d\alpha.
\label{eq:10.1}
\end{equation}
As mentioned already in the last section, $S_3(\alpha)$ consists
of only a bounded number of terms if $H$ is bounded:
the main term, corresponding to the
pole at $s = 1$ and possibly those arising from the generalized exceptional
zeros $\varrho$ with
\begin{equation}
\delta \leq H/\mathcal L, \quad |\gamma| \leq \sqrt{X}.
\label{eq:10.2}
\end{equation}
The number of the generalized exceptional zeros in (\ref{eq:10.2}) is by
Lemma~\ref{eq:4.18}
\begin{equation}
\leq Ce^{3H}
\label{eq:10.3}
\end{equation}
with an absolute constant $C$, where $H$ will be chosen as a large constant ($H = H(\eta_0)$) depending
on $\eta_0$.
(The value of $H$ will be determined later in the
next section.) In the following we will omit in our notation the dependence of the
constants on $\eta_0$. At any rate, if $\vartheta \leq 0.44$,
that is, $\eta_0 = 4/9 - 0.44$
for example, then all constants will be absolute constants.

Let $\varrho_0 = 1$ and $\varrho_\nu$ $(\nu = 1, \dots, M)$
denote the possible generalized exceptional zeros of $L(s, \chi_\nu)$ with
primitive characters $\chi_\nu$, possibly equal, belonging to
conductors $r_\nu$. Here $M = 0$ is naturally possible, in which
case we have only the main term corresponding to $\varrho_0 =
1$.
We list multiple zeros according to their multiplicity.
Similarly to (\ref{eq:6.4}) we obtain
\begin{align}
 &\sum_{q \leq P} \underset{a}{\sum\nolimits^\prime}
\int\limits_{\mathfrak M(q,a)} S^2_3(\alpha) e(-m\alpha) d\alpha
=\label{eq:10.4}\\
&= \sum^M_{\nu=0} \sum^M_{\mu=0} \sum_{\substack{q\leq P\\ [r_\nu,
r_\mu]\mid q}} A(\varrho_\nu) A(\varrho_\mu)
 c(\chi_\nu, \chi_\mu, q, m) \int\limits^{1/qQ}_{-1/qQ}
T_{\varrho_\nu} (\eta) T_{\varrho_\mu}(\eta)
e(-m\eta) d\eta.
\nonumber
\end{align}
Until now the value of $P$ could be arbitrary. However, if a
$P_0 = X^{\vartheta_0}$ $(\vartheta_0 = {4\over 9} - \eta_0$,
$\,\eta_0 > 0)$ is given, we will choose $P$ suitably within the
range ($\varepsilon' > 0$, sufficiently small)
\begin{equation}
P\in [P_0 X^{-\varepsilon'}, P_0]
\label{eq:10.5}
\end{equation}
as to satisfy the following conditions (with $\varepsilon_0 =
\varepsilon'/10(M+1)^2$):
\begin{equation}
\aligned
&\text{if $[r_\mu, r_\nu] \leq P$ then $[r_\mu, r_\nu] \leq
PX^{-\varepsilon_0}\ \  (\nu,\mu \in [0, M])$}\\
&\text{if $|\gamma_\nu| \leq {P\over r_\nu} X^{\ve_0}$ then $|\gamma_\nu|
\leq {P\over r_\nu} X^{-4\varepsilon_0}\ \ (\nu \in [0, M])$}
\endaligned
\label{eq:10.6}
\end{equation}

First we will show that the effect of singularity pairs $\ell, \mu$ satisfying
\begin{equation}
{P\over r_\ell} X^{\varepsilon_0} \leq |\gamma_\ell| \leq \sqrt{X}
\label{eq:10.7}
\end{equation}
will be negligible, namely
\begin{equation}
\ll {\mathfrak S}(m) X^{1-\varepsilon_0}
\label{eq:10.8}
\end{equation}
for any pair $(\ell, \mu)$ $(\mu = 0,1,\dots, M)$.
Namely, similarly to (\ref{eq:9.4}--\ref{eq:9.5}) we obtain by
Gallagher's lemma (Lemma~\ref{l:4.11})
\begin{align}
\int\limits^{1/r_\ell Q}_{-1/r_\ell Q} |T^2_{\varrho_\ell}
(\eta)|d\eta &\ll \int\limits^X_{-r_\ell Q/2}
\left|{1\over r_\ell Q} \sum^X_{\substack{n = X_1\\x < n < x + r_\ell Q/2}} n^{\varrho_\ell - 1}
\right|^2 dx \ll \label{eq:10.9}\\
&\ll X\cdot \left({X^{1-\delta_\ell}\over r_\ell
Q|\varrho_\ell|}\right)^2 \ll X^{1-2\delta_\ell - 2\varepsilon_0} \ll
X^{1-2\varepsilon_0} .
\nonumber
\end{align}
Since we have trivially by Parseval's identity for any $\mu$
\begin{equation}
\int\limits^{1/rQ}_{-1/rQ} |T^2_{\varrho_\mu}(\eta)|d\eta \leq
\int\limits^1_0 |T^2_{\varrho_\mu}(\eta)|d\eta = \sum^X_{n =
X_1} n^{-2\delta_\mu} \leq X,
\label{eq:10.10}
\end{equation}
the Cauchy--Schwarz inequality yields for all $q$ with any $\varrho_\ell$ in \eqref{eq:10.7} and
$[r_\ell, r_\mu]\mid q$
%%:1/qQ$
\begin{equation}
\int\limits^{1/qQ}_{-1/qQ} |T_{\varrho_\mu}(\eta)
T_{\varrho_\ell}(\eta)| d\eta \ll X^{1 - \varepsilon_0} .
\label{eq:10.11}
\end{equation}
This, together with (\ref{eq:10.3}), shows (\ref{eq:10.8}).

So we can reduce our attention to zeros $\varrho_\nu$ satisfying
\begin{equation}
|\gamma_\nu| \leq {P\over r_\nu} X^{-4\varepsilon_0}
\label{eq:10.12}
\end{equation}
and we can delete with an error of $O({\mathfrak S}(m)X^{1-\varepsilon_0})$ all others.
Let us denote the remaining zeros (satisfying (\ref{eq:10.12}))
by $\varrho_\nu$, $\nu = 1,\dots, K$.
Now (\ref{eq:10.12}) implies immediately
\begin{equation}
{|\gamma_\nu|\over X_1} = {|\gamma_\nu|X^{\varepsilon_0}\over
X} \leq {X^{-3\varepsilon_0}\over r_\nu Q} \leq {1 \over qQ} \
\text{ if } q \leq X^{3\varepsilon_0} r_\nu.
\label{eq:10.13}
\end{equation}

We will show the following

\begin{Pro}
Let $\varrho_\nu$ satisfy {\rm (\ref{eq:10.13})}. Then
\begin{equation}
T_{\varrho_\nu}(\eta) \ll (X^{\delta_\nu}_1 \eta)^{-1} \ll
|\eta|^{-1} \ \text{ if } {|\gamma_\nu|\over X_1} \leq |\eta| \leq
1/2.
\label{eq:10.14}
\end{equation}
\end{Pro}

\begin{proof}
Let us consider the trigonometric sum
\begin{equation}
U(\gamma_\nu, \eta, y) = \sum_{X_1 < n \leq y} n^{i\gamma_\nu}
e(n\eta) = \sum_{X_1 < n \leq y} e(f(n)), \quad
X_1 < y \leq X,
\label{eq:10.15}
\end{equation}
where
\begin{equation}
f(u) = {\gamma_\nu\over 2\pi} \log u + \eta u.
\label{eq:10.16}
\end{equation}
For $u\in [X_1, X]$ we clearly have $f'(u) = \eta - \gamma_\nu /
2\pi u$ monotonic, the same sign as $\eta$ and by $|\gamma_\nu|/u \leq |\gamma_\nu| / X_1
\leq |\eta|$ we have also
\begin{equation}
|\eta|/2 < |f'(u)| < 3|\eta|/2.
\label{eq:10.17}
\end{equation}
Thus Lemmas \ref{l:4.5} and \ref{l:4.6} give
\begin{equation}
U(\gamma_\nu, \eta, y) \ll |\eta|^{-1}.
\label{eq:10.18}
\end{equation}
Now (\ref{eq:10.14}) follows by partial summation.
\end{proof}

The above proposition implies $(\kappa = \pm 1)$ for any pair of
remaining singularities $\varrho_\nu, \varrho_\mu$ of $L'/L(s,\chi)$ $\,(\nu, \mu
\in [0, K]$, $q_0 = [r_\nu, r_\mu] \leq P)$ by (\ref{eq:10.6}), \eqref{eq:10.12}--\eqref{eq:10.13},
and the Main Lemma (cf.\ (\ref{eq:7.7}) and (\ref{eq:7.14}, $t = hk$)
\begin{equation}
\aligned
&\sum_{\substack{q\leq X^{3\varepsilon_0} \min(r_\nu, r_\mu),\ q\leq P\\
[r_\nu, r_\mu]\mid q}} |c(\chi _\nu, \chi_\mu, q, m)|
\int\limits^{\kappa/2}_{\kappa/qQ} |T_{\varrho_\nu}(\eta)| \,
|T_{\varrho_\mu}(\eta)|d\eta \ll\\
&\ll Q \sum_{\substack{h \leq P/q_0\\ h \mid m}} \sum_{k \leq
P/hq_0} {\mathfrak S}(m) \cdot {q_0 hk \over \varphi(h) \varphi^2(k)}
\ll Q[r_\nu, r_\mu] X^{\varepsilon_0/2} \ll X^{1 -
\varepsilon_0/2} .
\endaligned
\label{eq:10.19}
\end{equation}

Using the trivial estimate
\begin{equation}
\int\limits^1_0 |T_\varrho(\eta)T_{\varrho'}(\eta)|d\eta \leq
\bigg( \int\limits^1_0 |T^2_\varrho(\eta)|d\eta\bigg)^{1/2}
\bigg(\int\limits^1_0 |T^2_{\varrho'}(\eta)|d\eta\bigg)^{1/2} \leq X,
\label{eq:10.20}
\end{equation}
we obtain for the contribution of the terms with
\[
q > [r_\nu, r_\mu] X^{\varepsilon_0} = q_0 X^{\varepsilon_0},
\]
by \eqref{eq:7.14}, similarly to (\ref{eq:10.19}), the following bound:
\begin{equation}
\aligned
&X \sum_{t > X^{\varepsilon_0}} \bigl|c(\chi_\nu, \chi_\mu, q_0 t,m)\bigr| \ll X
{\mathfrak S}(m) \sum_{h\mid m} {1 \over \varphi(h)} \sum_{k\geq
X^{\varepsilon_0}/h} {1\over \varphi^2(k)}\\
&\ll X {\mathfrak S}(m) \sum_{h\mid m} {h\over \varphi(h)}
X^{-\varepsilon_0} \ll {\mathfrak S}(m) X^{1-\ve_0/2},
\endaligned
\label{eq:10.21}
\end{equation}
which is negligible.

However, if
\[
X^{3\ve_0} \min (r_\nu, r_\mu) \leq [r_\nu, r_\mu] X^{\ve_0}
\]
then in (\ref{eq:1.22}) we have
\[
\sqrt{U} \geq \max \left({r_\nu \over (r_\nu, r_\mu)}, {r_\mu
\over (r_\nu, r_\mu)} \right) \geq X^{2\ve_0}
\]
and consequently by the Corollary to the Main Lemma (cf.\ \eqref{eq:1.21}) we have
\[
\bigl|{\mathfrak S}(\chi_\nu, \chi_\mu, m)\bigr| \leq {\mathfrak S}(m) X^{-\ve_0}.
\]
This implies for the possible contribution of the intermediate
terms with
\[
X^{3\ve_0} \min(r_\nu, r_\mu) \leq q \leq [r_\nu, r_\mu] X^{\ve_0}
\]
similarly to (\ref{eq:10.19}) the estimate (cf.\ \eqref{eq:7.2} and \eqref{eq:7.7} in the Main Lemma)
\[
O\bigl(X {\mathfrak S}(m) X^{-\ve_0}\bigr) \ll {\mathfrak S}(m) X^{1 - \ve_0}.
\]
Summarizing, we have for all pairs $\nu, \mu \in [0, K]$:
\begin{equation}
\sum_{\substack{q \leq P\\ [r_\nu, r_\mu]\mid q}} |c(\chi_\nu,
\chi_\mu, q, m)| \int\limits^{\kappa/2}_{\kappa/qQ}
|T_{\varrho_\nu}(\eta) T_{\varrho_\mu}(\eta)|d\eta \ll {\mathfrak S}(m)
X^{1-\ve_0/2}.
\label{eq:10.22}
\end{equation}

Now (\ref{eq:10.22}) means that we can extend the integration on
the right-hand side of (\ref{eq:10.4}), for the remaining singularities $\varrho_\nu$,
$\varrho_\mu$ ($\nu, \mu = 0,1,\dots, K$) for the full interval
$[0,1]$ in place of $[-1/qQ, 1/qQ]$, with an error of size $O(\mathfrak S(m) X^{1-\ve_0/2})$. Here the full integral can be expressed by
the $\Gamma$-function (cf.\ Lemmas~\ref{l:4.8}--\ref{l:4.9}) as follows:
\begin{equation}
\int\limits^1_0 T_{\varrho_\nu}(\eta) T_{\varrho_\mu}(\eta)
e(-m\eta)d\eta
=  {\Gamma(\varrho_\nu) \Gamma(\varrho_\mu) \over
\Gamma(\varrho_\nu + \varrho_\mu)} m^{\varrho_\nu + \varrho_\mu
- 1} + O(X_1).
\label{eq:10.23}
\end{equation}

Further, as $[r_\nu, r_\mu) < P$ implies $[r_\nu, r_\mu] X^{\ve_0} < P$, the effect of all terms with $q \geq P$
is by \eqref{eq:10.21} negligible.
So, from (\ref{eq:10.1}), (\ref{eq:10.4}), (\ref{eq:10.8}),
(\ref{eq:10.19}), (\ref{eq:10.22}) and (\ref{eq:10.23}) we have
\begin{equation}
\aligned
&\intl_{\mathfrak M} S^2_3(\alpha) e(-m\alpha)d\alpha\\
&=
\underset{[r_\nu, r_\mu] \leq P}{\sum^K_{\nu = 0} \sum^K_{\mu = 0}}
{\mathfrak S}(\chi_\nu, \chi_\mu, m) A(\varrho_\nu) A(\varrho_\mu)
{\Gamma(\varrho_\nu) \Gamma(\varrho_\mu) \over
\Gamma(\varrho_\nu + \varrho_\mu)} m^{\varrho_\nu + \varrho_\mu
- 1} + O(X^{1 - \varepsilon_0/2}).
\endaligned
\label{eq:10.24}
\end{equation}

Since we have for the generalized exceptional singularities
\begin{equation}
{\Gamma(\varrho_\nu) \Gamma(\varrho_\mu) \over
\Gamma(\varrho_\nu + \varrho_\mu)} \ll \big(\max(|\gamma_\nu|,
|\gamma_\mu|)\big)^{-1/2}
\label{eq:10.25}
\end{equation}
we can further learn from our formula (\ref{eq:10.24}) that up
to an error of\break
$O({\mathfrak S}(m)X/ \sqrt{T_0})$, zeros of height
\begin{equation}
|\gamma| \geq T_0
\label{eq:10.26}
\end{equation}
may be neglected as well. If we have no Siegel zero, then the
error ${\mathfrak S}(m)X/\sqrt{T_0}$ will be admissible if we choose
$T_0$ as a large constant.
This can be seen from \eqref{eq:10.24} since we have our main term
corresponding to $(\varrho_0, \varrho_0) = (1,1)$ in the sum -- which yields ${\mathfrak S}(m) m$.
Therefore, in addition to (\ref{eq:10.24}) the following formula summarizes the results of this section:
\begin{align}
&\intl_{\mathfrak M}\! S^2_3(\alpha) e(-m\alpha)d\alpha \label{eq:10.27}\\
& =  \underset {[r_\nu, r_\mu] \leq P}
{\sum^K_{\substack{\nu = 0\\ |\gamma_\nu| \leq T_0}}
\sum^K_{\substack{\mu = 0\\ |\gamma_\mu| \leq T_0}}} \!
{\mathfrak S}(\chi_\nu, \chi_\mu,
m) A(\varrho_\nu) A(\varrho_\mu) {\Gamma(\varrho_\nu)
\Gamma(\varrho_\mu) \over \Gamma(\varrho_\nu + \varrho_\mu)}
m^{\varrho_\nu + \varrho_\mu - 1} \nonumber\\
&\quad + O({\mathfrak S}(m)X / \sqrt{T_0}) +
O(X^{1 - \ve_0 / 2}) .
\nonumber
\end{align}

The relations \eqref{eq:10.24}, \eqref{eq:10.27} actually prove the explicit formula \eqref{eq:2.10} if we take into account \eqref{eq:1.21}--\eqref{eq:1.22} which follows from the Main Lemma~1.
Therefore our Theorem~\ref{t:1} is proved.

\section{Proof of Theorem~2}
\label{s:11}

Suppose that after choosing suitably $P$ in
(\ref{eq:10.5})--(\ref{eq:10.6}) we have the following case:

There is a unique real primitive character $\chi_1 (\mod r_1)$, $r_1 \leq P$ such that $L(s, \chi_1)$ has a real
zero $\varrho_1 = \beta_1 = 1 - \delta_1$ with
\begin{equation}
\delta_1 \leq h /\log X = h/\mathcal L,
\label{eq:11.1}
\end{equation}
where $h$ is a constant, to be chosen at the end of the section
which may depend on~$\eta_0$.

We remark here that $h$ will be a small constant and the
constant $H$ (cf.\ (\ref{eq:2.8}), (\ref{eq:10.2})) will be a
large constant depending on $h$.
We also remark that by the procedure in
(\ref{eq:10.5})--(\ref{eq:10.6}) we will have actually $r_1 =
[1, r_1] \leq PX^{-\varepsilon_0}$.

In this case we will show, using the Deuring--Heilbronn phenomenon
(Lem\-ma~\ref{l:4.21}) that $S_3(\alpha)$ consists exactly of two
terms: those corresponding to $\varrho_0 = 1$ and the Siegel
zero $\varrho_1$ above. Further also some part of the region
\[
\mathcal R = \big\{s; \sigma \geq 1 - b, |t| \leq \sqrt{X}\big\}
\]
associated with the definition of terms in $S_2(\alpha)$ will be
free of zeros of $L(s, \chi, r)$ if $r \leq P$. The size
of the actual zero-free part of $\mathcal R$ will depend on
$\delta_1$, that is, how close the real zero $\varrho_1 =
\beta_1$ lies to~$1$.

We remark first that in case of an arbitrary primitive character
$\chi_2$ $\mod r_2 \leq P$ and $\varrho_2 = 1 - \delta_2 + i\gamma_2$
with
\begin{equation}
\delta_2 \leq {H\over \mathcal L}, \quad
|\gamma_2| \leq \sqrt{X},
\label{eq:11.2}
\end{equation}
we have in Lemma~\ref{l:4.21} the ``trivial'' estimate
\begin{equation}
Y = (r^2_1 r_2 k(|\gamma| + 2)^2)^{3\over 8} \ll (P^5
X)^{3\over 8} \leq X^{29\over 24},
\label{eq:11.3}
\end{equation}
where $k = \mathrm{cond}\, \chi_1 \chi_2$. This implies in case of $\delta_2 < 1/200$ by \eqref{eq:4.19}
\begin{equation}
h \geq \mathcal L \delta_1 \geq {24\over 29} \delta_1
\log Y >
{1\over 2} Y^{-{(1+\ve)\over 1-6\delta_2}\delta_2} > {1\over 2}
X^{-{5\over 4} \delta_2}.
\label{eq:11.4}
\end{equation}
From this we obtain
\begin{equation}
\delta_2 \mathcal L > {4\over 5} \log {1\over 2h} = H_0(h).
\label{eq:11.5}
\end{equation}
This means that choosing $H = H_0(h)$, the existence of a Siegel
zero will really imply that there are no other zeros in the
region~(\ref{eq:11.2}).

By $Y \geq r^{3/4}_1$, the ineffective theorem of Siegel (Lemma~\ref{l:4.14}) implies
for any $\ve_1$
\begin{equation}
\delta_1 \geq \max (P^{-\ve_1}, Y^{-\ve_1}) \
\text{ if } X \geq X_1(\ve),\ Y \geq Y_1(\ve).
\label{eq:11.6}
\end{equation}
So we have for $\delta < 1/200$ from (\ref{eq:4.20}) the inequality
\begin{equation}
{\delta_2 \over 1 - 6\delta_2} > (1 - \ve_1) {\log{2\over 3\delta_1
\log Y} \over \log Y}.
\label{eq:11.7}
\end{equation}

Since by (\ref{eq:11.6}) the right-hand side is here $< \ve_1$,
(\ref{eq:11.7}) implies that
\begin{equation}
\delta_2 > (1 - 7\ve_1) {\log{2\over 3\delta_1 \log Y}\over \log
Y} \overset{\mathrm{def}}{=} \varphi_0(Y)
\label{eq:11.8}
\end{equation}
for any $\ve_1$, $Y \geq Y(\ve_1)$. Here $\varphi_0(Y) < \ve$,
if $Y \geq Y_2(\ve, \ve_1)$.
Let us denote now (cf.\ \eqref{eq:11.3})
\begin{equation}
(\delta_1 \mathcal L)^{-1} = G_1 \geq h^{-1}, \quad
G(Y) = G = {2\over 3\delta_1 \log Y} > {G_1 \over 2}.
\label{eq:11.9}
\end{equation}
We recall (cf.\ \cite{Pin1}) that effectively $\delta_1 \gg
r^{-1/2}_1$, consequently $r_1 \gg \mathcal L^2$.
Further by the Main Lemma~1 (cf.\ \eqref{eq:7.5}--\eqref{eq:7.7} and \eqref{eq:1.21}--\eqref{eq:1.22})
\begin{eqnarray}
{\mathfrak S}(\chi_0, \chi_0, m) & = & {\mathfrak S}(m)\nonumber\\
{\mathfrak S}(\chi_1, \chi_0, m) & \ll & {r_1 \over \varphi^2(r_1)} {\mathfrak S}(m)
\ll {\mathfrak S}(m) {\log r_1 \over r_1},\\
|{\mathfrak S}(\chi, \chi', m)| & \leq & {\mathfrak S}(m) \ \text{ for any } \chi,\chi'. \nonumber
\label{eq:11.10}
\end{eqnarray}
Hence the asymptotic formula (\ref{eq:10.24}) tells us that
\begin{align}
&\intl_{\mathfrak M} S^2_3(\alpha) e(-m\alpha)d\alpha\label{eq:11.11}\\
&= {\mathfrak S}(m)m+{\mathfrak S}(\chi_1,\chi_1,m) {\Gamma(1 - \delta_1)^2 \over
\Gamma(2 - 2\delta_1)} m^{1-2\delta_1}\nonumber\\
&\quad + O\left({\mathfrak S}(m) m {\log r_1 \over r_1}\right) + O(X^{1 -
\ve_0/2}) \nonumber \\
&\geq {\mathfrak S}(m) m(1 - e^{-2\delta_1 \log m} + O(\delta_1) +
O(X^{-\ve_0/2})) \nonumber \\
& > 1.9 \cdot {{\mathfrak S}(m)m \over G_1}
\nonumber
\end{align}
if $m > X^{1-\ve_0/3}$, since the last error term is negligible
in view of Siegel's theorem (cf.\ (\ref{eq:11.6})).

Now, using the zero-free region (\ref{eq:11.8}) we try to show
that, possibly for all $[m \in X/2, X]$
\begin{equation}
\bigg|\intl_{\mathfrak M} \big(S^2_2(\alpha) + 2S_2
S_3(\alpha)\big) e (-m\alpha) d\alpha \bigg| \leq
{5{\mathfrak S}(m)m\over 6G_1}.
\label{eq:11.12}
\end{equation}

We remark that (\ref{eq:11.12}) is sufficient to show our
Theorem, in view of the notation \eqref{eq:6.2}--\eqref{eq:6.3}, the
final result of Section~8 (\ref{eq:8.28}) and $G^{-1}_1 \geq P^{-\ve}$ (cf.\
(\ref{eq:11.6})). It would actually be possible to show
(\ref{eq:11.12}) for all $m$ for some smaller value of
$\vartheta$ than $4/9$ $(\vartheta = 16/39)$. However, we can
also show for any $\vartheta < 4/9$ that (\ref{eq:11.12}) holds
for all but
\begin{equation}
O\left({X^{1+\ve} \over P}\right)
\label{eq:11.13}
\end{equation}
values $m \in [X/2, X]$ which is an admissible size exceptional
set; the same as or better than the cardinality of the exceptional set arising
from the minor arcs (cf.\ (\ref{eq:5.3})).

Let us investigate now $\int(S^2_2(\alpha) +
2S_2(\alpha)S_3(\alpha))e(-m\alpha )d\alpha$. The number of
zeros appearing in $S_0(\alpha)$ is by Lemma~\ref{l:4.17}
\begin{equation}
\ll (P^2 \sqrt X)^{(2+\ve)b} \ll X^{3b}
\label{eq:11.14}
\end{equation}
and $b$ can be chosen arbitrarily small. The number of pairs of
zeros is consequently $\ll X^{6b}$.

In the present section we will suppose $b < b_1(\eta_0)$ a fixed
constant, whose value will be determined later. First we can
observe that the total contribution of all zeros $\varrho$ of all
$L(s, \chi)$ belonging to primitive characters $\chi \mod r$ with
\begin{equation}
0 \leq \delta \leq b, \quad |\gamma| > {P\over R} X^b, \quad
r\in [R, RX^b], \quad R \leq P
\label{eq:11.15}
\end{equation}
to $\sum\sum W'_2(x)$ in (\ref{eq:9.11}) is -- according to the
argumentation in (\ref{eq:9.11}) -- for $b\leq 1/8$:
\begin{equation}
\ll X^{1/2} \sum_{\substack{\mu\\ 2^\mu \geq X^b}}
(2^\mu)^{-1+bc_2(b)(1 + \ve)} e^{-cH} \ll X^{1/2-b/2}.
\label{eq:11.16}
\end{equation}
This implies for their total contribution to $\int S^2_2 + 2S_2
S_3$, by (\ref{eq:9.11}) and (\ref{eq:9.13}), the estimate
\begin{equation}
%%\sigma(m) X^{1/2} \cdot e^{-cH} \cdot X^{1/2 - b/2}
\ll {\mathfrak S}(m) X^{1-b/2} = o \left({{\mathfrak S}(m)\cdot m \over G}\right)
\label{eq:11.17}
\end{equation}
(naturally $S_3(\alpha)$ is completely known by Section~\ref{s:10}
explicitly, since it has now only the two terms $\varrho_0 = 1$,
$\varrho_1 = 1 - \delta_1$).

So we will suppose from now on, in this section, that
\begin{equation}
0 \leq \delta \leq b, \quad |\gamma| \leq {P\over R} X^b,\quad
r\in [R, RX^b]
\label{eq:11.18}
\end{equation}
for the zeros associated with $S_2(\alpha)$ (cf.\ \eqref{eq:6.2}--\eqref{eq:6.3}).

Further we can suppose that for the given $m$ we have for all
$\chi_\nu(\mod r_\nu)$, $\chi_\mu(\mod r_\mu)$
\begin{equation}
|{\mathfrak S}(\chi_\nu, \chi_\mu, m)| \geq X^{-b} {\mathfrak S}(m)
\label{eq:11.19}
\end{equation}
since for the total contribution of pairs not satisfying
(\ref{eq:11.19}) we have directly by (\ref{eq:9.2}) and
(\ref{eq:9.11})--(\ref{eq:9.13}) the estimate
\begin{equation}
B{\mathfrak S}(m) X^{-b} \Biggl\{\! \biggl(\underset{r(\chi)\leq
P}{\sum\nolimits^*} W_2(\chi)\biggr)^{\! 2}\! +\! \sum_{r(\chi)\leq P}\
\underset{r(\chi')\leq P}{\sum\nolimits^*} W_2(\chi) W_3(\chi')\! \Biggr\}
\ll
X^{1-b} {\mathfrak S}(m).
\label{eq:11.20}
\end{equation}
However, according to the Main Lemma, (\ref{eq:11.19}) implies (see
\eqref{eq:1.21}--(\ref{eq:1.22}))
\begin{equation}
U(\chi_\nu, \chi_\mu, m) \ll X^{3b} .
\label{eq:11.21}
\end{equation}

In what follows we will delete pairs in $S^2_2$ contradicting to \eqref{eq:11.19}.
(\ref{eq:11.21}) implies also
\begin{equation}
X^{-3b/2} \ll r_\nu / r_\mu \ll X^{3b/2}.
\label{eq:11.22}
\end{equation}

Let us consider first the easier case $S_3 \cdot S_2$. In this case the
term $(\varrho_j, \chi_j, r_j)$ coming from $S_3$ is either
\begin{equation}
(1, \chi_0, 1) \ \text{ or } \ (1 - \delta_1, \chi_1, r_1) \quad (j =
0 \text{ or } 1).
\label{eq:11.23}
\end{equation}
Let us suppose first that $j = 1$.
If for the term $(\varrho, \chi,r)$ (\ref{eq:11.19}) is false we can delete it.
So we can suppose here by \eqref{eq:1.21}--\eqref{eq:1.22} that we have for all $(\varrho,
\chi,r)$ in $S_2$
\begin{equation}
r_j X^{-3b/2} \ll r \ll r_j X^{3b/2}, \quad
\mathrm{cond}\, \chi \chi_j \ll X^{3b}
\label{eq:11.24}
\end{equation}
at least for the examination of $S_2 \cdot S_{31}$, the part
coming from $\chi_1$.
Thus we have for any pair $\chi, \chi'$ of characters remaining in
$S_2$ after the deletion
\begin{equation}
\mathrm{cond}\, \chi \overline \chi' \leq \mathrm{cond} \, \chi_1 \chi \cdot
\mathrm{cond}\, \overline{\chi_1 \chi'} \leq X^{6b}.
\label{eq:11.25}
\end{equation}
Let us denote the corresponding new set by $S'_{21}$.
Now we are able to use our density Lemma~\ref{l:4.20}, more exactly \eqref{eq:4.18}.
If the constant $b$ is
chosen sufficiently small in dependence on $\eta_0$ we have for
any $R_\nu = X^{\nu b} \leq P$, $\delta \leq b$
\begin{align}
& \sum_{\substack{R_\nu < r_\nu < R_\nu X^b\\
r_\nu \leq P}}\ \underset{\chi_\nu(r_\nu) \in
S'_{21}}{\sum \nolimits^*} N \left(1 - \delta, {P\over R_\nu}
X^b, \chi_\nu\right) \label{eq:11.26} \\
&\ll_b (PX^{18 b})^{(3/4 + \sqrt[3]{b})\delta} \ll_b
P^{(3/4 + 2\sqrt[3]{b})\delta} \ll_{\eta_0, b} X^{\delta/3}.
\nonumber
\end{align}
Thus we obtain by the Deuring--Heilbronn phenomenon
(\ref{eq:11.8}), similarly to (\ref{eq:9.11})
\begin{equation}
S^*_{231} \overset{\rm def}{=} X^{-{1\over2}} \sum_{r_\nu \leq P} \sum_{\chi_\nu(r_\nu)\in S'_{21}}
W'_2 (\chi) \ll_{\eta, b} X^{-(2/3)\varphi_0(Y)}.
\label{eq:11.27}
\end{equation}

Now we will show an estimate sharper than (\ref{eq:11.3})
for~$Y$.
In view of (\ref{eq:11.22}) and (\ref{eq:11.18}) we have for any pair $(\varrho,\chi)$, $\chi\mod r$, remaining in $S'_{21}$
\begin{equation}
Y = \big(r^2_1 rk(|\gamma| + 2)^2\big)^{3/8} \ll \left( r^3 X^{6b} \left({P \over r} X^{2b}\right)^2\right)^{3\over 8} \ll
P^{9\over 8} X^{4b} \leq  \sqrt X .
\label{eq:11.28}
\end{equation}
Substituting this into (\ref{eq:11.27}) we obtain
\begin{align}
S^*_{231}
&\ll_{\eta_0, b} \exp \left( -{2\over
3} \mathcal L {(1 - 7\ve_1) \log G(\sqrt{X})\over \mathcal L/2} \right)
\ll_{\eta_0,b} G(\sqrt{X})^{-4/3 + 10\varepsilon_1}\label{eq:11.29}\\
& \ll_{\eta_0,b} G(\sqrt{X})^{-\frac54}
\nonumber
\end{align}
if $\ve_1 < 1/120$.
Now let us fix a small $b$ in dependence on $\eta_0$. If now $h$
is chosen small enough in dependence on $\eta_0$, then $G(\sqrt{X}) = 4G_1/3 \geq 4h_1^{-1} /3$
will be sufficiently large in dependence on $\eta_0$, and so we obtain from \eqref{eq:11.29} finally
\begin{equation}
\bigg|\intl_{\mathfrak M} S_2(\alpha) S_{31}(\alpha) e(-m\alpha)
d\alpha \bigg| < {{\mathfrak S}(m)X \over
12 G_1} < {{\mathfrak S}(m)m\over 6G_1}.
\label{eq:11.30}
\end{equation}
If we take $\varrho = \varrho_0 = 1$, $\chi_0$, $ r_0 = 1$, then we
have by (\ref{eq:11.24}) for the undeleted terms
\begin{equation}
r \ll X^{3b/2}.
\label{eq:11.31}
\end{equation}

Consequently, using \eqref{eq:4.18a} in place of \eqref{eq:4.18} we obtain the improved estimate
$X^{10b\delta}$ in place of \eqref{eq:11.26}.
Accordingly we obtain the estimate
$X^{-(1 - 10b)\varphi_0(Y)}$ instead of \eqref{eq:11.27}.
Further,
\begin{equation}
Y = \bigl(r^2_1 rk(|\gamma| + 2)^2\bigr)^{3/8} \ll \left(P^3 r^2 \left(\frac{P}{r} X^{2b}\right)^2\right)^{3/8}
= P^{\frac{15}{8}} X^{\frac{3b}{2}} \leq X^{\frac56}.
%%\tag{11.31a}
\label{eq:11.31a}
\end{equation}
Therefore we obtain, similarly to \eqref{eq:11.29} for the analogous quantity $S^*_{230}$
\begin{equation}
S^*_{230} \ll_{\eta_{0,b}} \exp \left( - (1 - 10b) \mathcal L \frac{(1 - 7\varepsilon_1)\log G(X^{5/6})}{5\mathcal L/6} \right) \ll_{\eta_{0, b}} G(X^{5/6})^{-7/6}.
%%\tag{11.31b}
\label{eq:11.31b}
\end{equation}
Therefore we have by \eqref{eq:11.9} also
\begin{equation}
\bigg|\intl_{\mathfrak M} S_2(\alpha) S_{30}(\alpha) e(-m\alpha)
d\alpha \bigg| < {{\mathfrak S}(m)m\over 6G_1},
\label{eq:11.32}
\end{equation}
where $S_{30}$ denotes the part of $S_3$ corresponding to
$\varrho_0 = 1$.

In order to treat $\int S^2_2(\alpha)$ let us consider any fixed
pair $(\chi_j, \varrho_j) \mod r_j \in [R, R X^b]$ and consider the
set $\mathcal S(\varrho_j, \chi_j)$ of all pairs $(\varrho_\mu,
\chi_\mu)$ $\chi_\mu \mod r_\mu$ in $S_{2j}$ for which
(\ref{eq:11.19}) and therefore (\ref{eq:11.21}),
(\ref{eq:11.22}) and (\ref{eq:11.24}) hold (with $\varrho =
\varrho_j$).
By symmetry we can suppose $\delta_\mu \geq \delta_j$.

The upper estimate for all possible $Y = Y(\varrho')$,
$(\varrho', \chi') \in \mathcal S(\varrho_j, \chi_j)$ will be now,
again by (\ref{eq:11.18}), (\ref{eq:11.22}), (\ref{eq:11.24}),
similarly to (\ref{eq:11.28}):
\begin{equation}
Y = (r^2_1 r_\mu k (|\gamma_\mu| + 1)^2)^{3/8} \ll
r^{3/4}_1 k^{3/8} R^{3/8} \left({P\over R}\right)^{3/4} X^{2b}
\ll P^{3/2} k^{3/8} R^{-{3\over 8}} X^{2b}.
\label{eq:11.33}
\end{equation}
If we would like to have an estimate, valid for all $m \in [X/2,
X]$ we can estimate $k$ by $P^2$ from above and obtain
\begin{equation}
Y \ll P^{9/4} R^{-{3\over 8}} X^{2b} =: Z.
\label{eq:11.34}
\end{equation}

Further, due to $\delta_\mu \geq \delta_j$ we obtain, as in
(\ref{eq:11.26})--(\ref{eq:11.27})
\begin{equation}
X^{-1/2} \sum_{\chi_\mu} W''_2(\chi_\mu) \ll \big(P^{{3\over 4}+2\sqrt[3]{b}}
X^{-1}\big)^{\delta_j},
\label{eq:11.35}
\end{equation}
where the summation runs over all $\chi_\mu$ for which there exists
$\varrho_\mu$ with $\varrho_\mu$, $\chi_\mu \in \mathcal S(\varrho_j,
\chi_j)$.

On the other hand, the contribution of all pairs $(\chi_j,
\varrho_j)$ with $\chi_j \mod r_j \in [R, RX^b]$, $|\gamma_j| \leq
{P\over r_j} X^b \leq {P\over R} X^b$ to $\sum\sum W(\chi_j)$ is,
multiplied by (\ref{eq:11.35}), similarly to (\ref{eq:9.11})
\begin{align}
&\ll \mathcal L \intl^b_{\varphi_0(Z)} \big(R^{c_1(b) - c_2(b)}
P^{c_2(b)} X^{-1+6b} P^{{3\over 4} + 2 \sqrt[3]{b}}
X^{-1}\big)^\delta d\delta \label{eq:11.36}\\
&\ll_{\eta,b} \big(R^{3/4} P^{9/4 + 3\sqrt[3]{b}}
X^{-2}\big)^{\varphi(Z)} .
\nonumber
\end{align}

Let $u = \log R / \mathcal L(\leq \vartheta - \eta_0)$. Then the
above estimate is by (\ref{eq:11.8}) and (\ref{eq:11.34})
\begin{equation}
\leq c(\eta_0,b) \exp \left( - \left( {2 - (9/4)\vartheta - (3/4)u
\over (9/4)\vartheta - (3/8)u} - \eta_0\right) \log G\right)
\leq c(\eta_0, b) G^{-1 - \eta_0}
\label{eq:11.37}
\end{equation}
if now, exceptionally $P \leq X^{\vartheta - \eta}$ with
$\vartheta = 16/39 < 4/9$, and $b$ is small enough in dependence on
$\eta$.

In this way we get analogously to (\ref{eq:11.30})--(\ref{eq:11.32})
\begin{equation}
\bigg| \intl_{\mathfrak M} S^2_2(\alpha) e(-m\alpha)
d\alpha\bigg| < {{\mathfrak S}(m)m\over 2G_1}.
\label{eq:11.38}
\end{equation}

In order to reach $\vartheta = {4\over 9}$, we need a further
idea. First we can remark that according to the Main Lemma~1 we have
\begin{equation}
|{\mathfrak S}(\chi_1, \chi_1, m)| \leq (\sqrt{3}/2) {\mathfrak S}(m), \quad
\text{if } r_1 \nmid 36m.
\label{eq:11.39}
\end{equation}
In this case the effect of the Siegel zero cannot destroy the main
term. Therefore in this case, according to Sections
\ref{s:8}--\ref{s:10}
\begin{align}
\intl_{\mathfrak M} S^2(\alpha)e(-m\alpha)d\alpha
&= \intl_{\mathfrak M} S^2_3(\alpha) e(-m\alpha)d\alpha +
O(e^{-cH}\cdot X) \label{eq:11.40}\\
&\geq {1\over 8} m{\mathfrak S}(m) + O(e^{-cH} X) > {1\over 9} m{\mathfrak S}(m),
\nonumber
\end{align}
and we are ready without any further analysis.
So we can suppose further on that $r_1 \mid 36m$.
In the argumentation (\ref{eq:11.33})--(\ref{eq:11.38}) we are allowed
to suppose (\ref{eq:11.21}), consequently,
\begin{equation}
g_\mu(m) = {r_\mu \over(r_\mu, m)} \ll X^{3b}.
\label{eq:11.41}
\end{equation}

Now we can distinguish two cases.

\subsection*{Case A. $[r_1, r_\mu] \leq P$.}

In this case we have $k = \mathrm{cond}\, \chi_1 \chi_\mu \leq P$, and
from (\ref{eq:11.33}) we have now
\begin{equation}
Y \ll P^{15/8} R^{-3/8} X^{4b} =: Z.
\label{eq:11.42}
\end{equation}
We have this in place of (\ref{eq:11.34}). (\ref{eq:11.35}) and
(\ref{eq:11.36}) remain true, whereas we have now instead of
(\ref{eq:11.37}) the final estimate
\begin{equation}
c(\eta,b) \exp \left\{ - \left({2 - (9/4) \vartheta -
(3/4)u\over (15/8) \vartheta - (3/8)u} - \eta \right) \log
G\right\} \leq c(\eta, b) G^{-1 - \eta}
\label{eq:11.43}
\end{equation}
if $u \leq \vartheta = 4/9 - \eta$, that is $R \leq P = X^{4/9 -
\eta}$.

\subsection*{Case B. $[r_1, r_\mu] > P$.}

Let us denote by $d_{\mu k}$ the divisors of $r_\mu$ with
$d_{\mu k} \leq X^{3b}$. If we consider any fixed pair $r_\mu,
d_{\mu k}$ then let us consider the set
\begin{equation}
M(r_\mu, d_{\mu k}) = \left\{36m;\  X/2 \leq m \leq X;\ r_1\mid 36m, \ {r_\mu
\over (r_\mu, m)} = d_{\mu k} \right\} .
\label{eq:11.44}
\end{equation}
Since $r_1 \mid 36m$, ${r_\mu\over d_{\mu k}} |m| 36m$, all
elements of $M(r_\mu, d_{\mu k})$ are multiples of
\begin{equation}
\left[ r_1, {r_\mu \over d_{\mu k}} \right] > {P\over d_{\mu k}}
\geq PX^{-3b},
\label{eq:11.45}
\end{equation}
so
\begin{equation}
|M(r_\mu, d_{\mu k})| \ll {X^{1+3b}\over P}.
\label{eq:11.46}
\end{equation}
The number of all moduli is by (\ref{eq:11.14}) $\ll X^{3b}$, so
the number of all pairs $r_\mu$, $d_{\mu k}$ is clearly $\ll
X^{6b}$.

Thus, throwing away all $m$'s with
\begin{equation}
\mathcal M = \left\{ m;\ 36m \in \bigcup_{r_\mu, d_{\mu k}}
M(r_\mu, d_{\mu k})\right\}
\label{eq:11.47}
\end{equation}
the cardinality of the arising new exceptional set will be
\begin{equation}
|\mathcal M| \leq {X^{1+9b} \over P}.
\label{eq:11.48}
\end{equation}
For all $m \in [X/2, X]\setminus \mathcal M$
we have Case A and therefore we obtain, by (\ref{eq:11.43}),
similarly to (\ref{eq:11.38}) (with $\vartheta = 4/9$)
\begin{equation}
\bigg|\intl_{\mathfrak M} S^2_2(\alpha) e(-m\alpha)d\alpha\bigg|
< {{\mathfrak S}(m)m\over 2 G_1}.
\label{eq:11.49}
\end{equation}

This, together with (\ref{eq:11.30}) and (\ref{eq:11.32}), really shows (\ref{eq:11.12}).
So by (\ref{eq:11.11}) we have
\begin{equation}
\intl_{\mathfrak M} S^2_0(\alpha) e(-m\alpha) d\alpha \geq
{1.05{\mathfrak S}(m)m\over G_1}.
\label{eq:11.50}
\end{equation}
Hence, inequalities (\ref{eq:11.50}) and (\ref{eq:8.28}) prove in
case of the existence of a Siegel zero
\begin{equation}
R_1(m) \geq {1.05 {\mathfrak S}(m) m \over G_1} + O(\mathcal L^8 X^{1 -
b/82}) > {\mathfrak S}(m) m \delta_1\mathcal L,
\label{eq:11.50a}
\end{equation}
in view of Siegel's theorem (\ref{eq:11.6})\,and\,(\ref{eq:11.9}),
for all values of $m \! \in\! [X/2, X] \backslash  \mathcal M$, where
the exceptional set $\mathcal M$ satisfies (\ref{eq:11.48}). The
constant $b$ can be chosen arbitrarily small here. Thus
(\ref{eq:11.50a}) proves our Theorem~\ref{t:2}.

\section{Conclusion}
\label{s:12}

In what follows we will investigate the sum
\begin{equation}
I'(\varrho_1, \varrho_2, m) = \sum_{X_2 < k \leq X - m}
k^{\varrho_1 - 1} (k + m)^{\overline \varrho_2 - 1}, \quad
(X_2 = X/4)
\label{eq:12.1}
\end{equation}
or, more precisely, first
\begin{equation}
J(\gamma_1, \gamma_2, m, u) = \sum_{X_2 < k \leq u} e\left(\frac{f(k)}{2\pi}\right)
\qquad (u \leq X - m)
\label{eq:12.2}
\end{equation}
where
\begin{equation}
f(y) = \gamma_1 \log y - \gamma_2 \log(y + m), \quad m \in [X/4, X/2], \ X/4 \leq y \leq X - m.
\label{eq:12.3}
\end{equation}
By symmetry we can clearly suppose $\gamma_1 \geq 0$. Let
\begin{equation}
M = \max(|\gamma_1|, |\gamma_2|) > C,
\label{eq:12.4}
\end{equation}
a suitably chosen large constant. With the aim to apply
Lemma~\ref{l:4.4} we calculate $f'(y)$:
\begin{equation}
f'(y) = {\gamma_1 \over y} - {\gamma_2\over y + m} = {\gamma_1 m
- (\gamma_2 - \gamma_1)y \over y(y + m)}.
\label{eq:12.5}
\end{equation}
We have
\begin{equation}
\alignedat2
f'(y) &\geq {\gamma_1 + |\gamma_2|\over X} \geq {M\over X} &
&\text{if } \gamma_2 \leq 0,\\
f'(y) &\geq {\gamma_1\over 4X} = {M\over 4X} & &\text{if } 0
\leq \gamma_2 \leq \gamma_1, \\
f'(y) &\geq {4\gamma_1 \over 3(y + m)} - {7\gamma_1 \over 6(y +
m)} \geq {\gamma_2 \over 7X} = {M \over 7X} \quad & &\text{if }
\gamma_1 \leq \gamma_2 \leq(7/6)\gamma_1.
\endalignedat
\label{eq:12.6}
\end{equation}
Thus, let us suppose $\gamma_2 > (7/6)\gamma_1$, $\gamma_2 = M >
C$ further on. In this case we have
\begin{equation}
f'(y)
\,\aligned >\\[-3.5mm]=\\[-3.5mm]<\endaligned \,
0 \quad \text{ if } \quad y
\, \aligned <\\[-3.5mm]=\\[-3.5mm]>\endaligned \,
{m\gamma_1 \over \gamma_2 - \gamma_1}.
\label{eq:12.7}
\end{equation}
Let $D = \sqrt M = \sqrt{\gamma_2}$. Now
\begin{equation}
f'(y) >  {DX \over y(y + m)} > {D\over X} \quad \text{if } y <
{m\gamma_1 - DX\over \gamma_2 - \gamma_1}
\label{eq:12.8}
\end{equation}
and
\begin{equation}
f'(y) < - {DX \over y(y + m)} < - {D\over X} \quad \text{if } y >
{m\gamma_1 + DX\over \gamma_2 - \gamma_1}.
\label{eq:12.9}
\end{equation}
So we can apply Lemma~\ref{l:4.4} if
\begin{equation}
y \notin \left[ {m\gamma_1 \over \gamma_2 - \gamma_1} - {DX
\over \gamma_2 - \gamma_1}, {m\gamma_1 \over \gamma_2 -
\gamma_1} + {DX \over \gamma_2 - \gamma_1}\right] = I_0.
\label{eq:12.10}
\end{equation}
Estimating the sum in (\ref{eq:12.2}) trivially if $k \in I_0$,
and otherwise by Lemma~\ref{l:4.4}, we obtain by $\gamma_2 - \gamma_1 > M/7$
\begin{equation}
J(\gamma_1, \gamma_2, m, u) \ll {X\over D} + {DX\over M} \ll
{X\over \sqrt M}.
\label{eq:12.11}
\end{equation}
Finally, by partial summation, (\ref{eq:12.11}) implies
\begin{equation}
I'(\varrho_1, \varrho_2, m) \ll {X^{1 - \delta_1 -
\delta_2}\over \sqrt{\max(|\gamma_1|, |\gamma_2|)}}.
\label{eq:12.12}
\end{equation}
The above estimate holds trivially if \eqref{eq:12.4} is false.
Thus we obtain an estimate, similar to (\ref{eq:10.25}), in case
of the Generalized Twin Prime Problem, too.
Theorem~\ref{t:1} is therefore completed by the above arguments and
by the results of Sections~\ref{s:8}--\ref{s:10}, more precisely
by (\ref{eq:8.28}), (\ref{eq:9.15}) and (\ref{eq:10.27}).

\noindent
J\'anos Pintz\\
R\'enyi Mathematical Institute\\
of the Hungarian Academy of Sciences\\
Budapest, Re\'altanoda u. 13--15\\
H-1053 Hungary\\
e-mail: pintz.janos@renyi.mta.hu

\end{document}